\documentclass[11pt, reqno]{amsart}

\usepackage{xcolor}
\usepackage{latexsym}
\usepackage{amssymb}
\usepackage{mathrsfs}
\usepackage{amsmath}
\usepackage{fancybox,color}
\usepackage{enumerate}
\usepackage[latin1]{inputenc}
\usepackage{mathtools}
\DeclarePairedDelimiter{\ceil}{\lceil}{\rceil}
\usepackage{eurosym}
\usepackage{url}
\usepackage{hyperref}
  \hypersetup{colorlinks=true, %
   citecolor=magenta,%
   linkcolor=magenta}

\newcommand{\kom}[1]{}
\renewcommand{\kom}[1]{{\bf [#1]}}

 \def\1{\raisebox{2pt}{\rm{$\chi$}}}

\newtheorem{theorem}{Theorem}[section]
\newtheorem{corollary}[theorem]{Corollary}
\newtheorem{lemma}[theorem]{Lemma}
\newtheorem{proposition}[theorem]{Proposition}
\newtheorem{definition}[theorem]{Definition}
\newtheorem{remark}[theorem]{Remark}

 \def\1{\raisebox{2pt}{\rm{$\chi$}}}
 
\def\vint_#1{\mathchoice%
          {\mathop{\kern 0.2em\vrule width 0.6em height 0.69678ex depth -0.58065ex
                  \kern -0.8em \intop}\nolimits_{\kern -0.4em#1}}%
          {\mathop{\kern 0.1em\vrule width 0.5em height 0.69678ex depth -0.60387ex
                  \kern -0.6em \intop}\nolimits_{#1}}%
          {\mathop{\kern 0.1em\vrule width 0.5em height 0.69678ex
              depth -0.60387ex
                  \kern -0.6em \intop}\nolimits_{#1}}%
          {\mathop{\kern 0.1em\vrule width 0.5em height 0.69678ex depth -0.60387ex
                  \kern -0.6em \intop}\nolimits_{#1}}}
\def\vintslides_#1{\mathchoice%
          {\mathop{\kern 0.1em\vrule width 0.5em height 0.697ex depth -0.581ex
                  \kern -0.6em \intop}\nolimits_{\kern -0.4em#1}}%
          {\mathop{\kern 0.1em\vrule width 0.3em height 0.697ex depth -0.604ex
                  \kern -0.4em \intop}\nolimits_{#1}}%
          {\mathop{\kern 0.1em\vrule width 0.3em height 0.697ex depth -0.604ex
                  \kern -0.4em \intop}\nolimits_{#1}}%
          {\mathop{\kern 0.1em\vrule width 0.3em height 0.697ex depth -0.604ex
                  \kern -0.4em \intop}\nolimits_{#1}}}

\newcommand{\kint}{\vint}

\newcommand{\aveint}[2]{\mathchoice%
          {\mathop{\kern 0.2em\vrule width 0.6em height 0.69678ex depth -0.58065ex
                  \kern -0.8em \intop}\nolimits_{\kern -0.45em#1}^{#2}}%
          {\mathop{\kern 0.1em\vrule width 0.5em height 0.69678ex depth -0.60387ex
                  \kern -0.6em \intop}\nolimits_{#1}^{#2}}%
          {\mathop{\kern 0.1em\vrule width 0.5em height 0.69678ex depth -0.60387ex
                  \kern -0.6em \intop}\nolimits_{#1}^{#2}}%
          {\mathop{\kern 0.1em\vrule width 0.5em height 0.69678ex depth -0.60387ex
                  \kern -0.6em \intop}\nolimits_{#1}^{#2}}}

\newcommand{\dist}{\operatorname{dist}}

\newcommand{\midrg}{\operatornamewithlimits{midrange}}

\numberwithin{equation}{section}
\allowdisplaybreaks

\definecolor{color1}{rgb}{0.309, 0.43,0.258}
\definecolor{color2}{rgb}{0.741, 0.502,0.743}
\definecolor{color3}{rgb}{0.580, 0.163,0.107}
\definecolor{color0}{rgb}{0.837, 0.309, 0.241}
\begin{document}

\title[]{Time-dependent tug-of-war games and normalized parabolic $p$-Laplace equations}

\author[Han]{Jeongmin Han}
\address{Department of Mathematical Sciences, Seoul National University 
1, Gwanak-ro, Gwanak-gu, Seoul, Republic of Korea}
\email{hanjm9114@snu.ac.kr}



\begin{abstract}
This paper concerns value functions of time-dependent tug-of-war games.
We first prove the existence and uniqueness of value functions
and verify that these game values satisfy a dynamic programming principle.
Using the arguments in the proof of existence of game values, we can also deduce asymptotic behavior of game values when $T \to \infty$.
Furthermore, we investigate boundary regularity for game values.
Thereafter, based on the regularity results for value functions, we deduce that game values converge to viscosity solutions of the normalized parabolic $p$-Laplace equation.
\end{abstract}

\maketitle

\tableofcontents

\section{Introduction}
Attempts to understand partial differential equations in relation to stochastic processes have been developed for several decades.
These approaches give us different perspectives on studying PDEs.
For instance, we can derive Laplace equation from standard Brownian motion and its transition semigroup (see, for example, \cite{MR3234570}).
This is a simple and classic example, but similar approaches have been suggested for more general equations as well.
Here we consider a generalization of the Laplace equation.
The normalized version of the $p$-Laplace operator can be written as
$$ \Delta_{p}^{N} u := \Delta u +(p-2)\Delta_{\infty}^{N} u =\Delta u + (p-2)\frac{\langle D^{2}uDu, Du\rangle}{|Du|^{2}}.$$
To deal with this nonlinear operator, we need to employ another diffusion process which is called a tug-of-war game. 
This stochastic game is a control interpretation associated with the $\infty$-Laplacian.
By adding noise to the tug-of-war game, we can adopt a stochastic view for the normalized $p$-Laplace equation.

In this paper, we study value functions of time-dependent tug-of-war games with noise.
In particular, we investigate several properties of game values such as regularity and long-time asymptotics.
Moreover, we also present uniform convergence of value functions to viscosity solutions of the normalized parabolic $p$-Laplace equation with $1 < p< \infty$ as the step size of game goes to zero.

Our study is the parabolic counterpart of \cite{MR3471974}.
In that paper, the author proved the existence, uniqueness and continuity of value functions for time-independent tug-of-war games.
We extend these results for game values to the case of time-dependent games and show that 
our value functions pointwisely converges to game values in \cite{MR3471974}
as $T \to \infty$.
On the other hand, we also establish uniform convergence of game values as $\epsilon \to 0$.
For this purpose, we need suitable regularity results.
We already showed interior regularity estimates for game values in \cite{MR4153524}.
In this paper, we also give boundary regularity results for value functions.
Actually, the main difficulty occurs in this part.  
Under the settings of \cite{MR3011990}, one can obtain the boundary estimates of value functions by using the exit time of the noise-only process.
Unfortunately, the noise in our settings depends on strategies of each player.
Thus, we establish the desired estimates using another appropriate stochastic game.

To derive our desired results, we employ the following DPP
\begin{align} \begin{split} \label{dppvar}
& u_{\epsilon}(x,t) \\ & = \frac{1- \delta(x,t)}{2} \times \\ & 
\bigg[ \hspace{-0.2em} \sup_{\nu \in S^{n-1}} \bigg\{ \alpha u_{\epsilon} \bigg(x+ \epsilon \nu ,t-\frac{\epsilon^{2}}{2} \bigg)  \hspace{-0.3em}+ \hspace{-0.3em} \beta \kint_{B_{\epsilon}^{\nu} }u_{\epsilon} \bigg(x+ h,t-\frac{\epsilon^{2}}{2} \bigg) d \mathcal{L}^{n-1}(h) \bigg\}   \\ &   \hspace{-0.3em}
+ \hspace{-0.3em} \inf_{\nu \in S^{n-1}} \bigg\{ \alpha u_{\epsilon} \bigg(x+ \epsilon \nu ,t-\frac{\epsilon^{2}}{2} \bigg)  \hspace{-0.3em}+ \hspace{-0.3em} \beta \kint_{B_{\epsilon}^{\nu} }u_{\epsilon} \bigg(x+ h,t-\frac{\epsilon^{2}}{2} \bigg) d \mathcal{L}^{n-1}(h) \bigg\} \bigg] \\ & 
 + \delta(x,t) F(x, t),
\end{split}
\end{align} 
where $0<\alpha, \beta <1$ with $\alpha + \beta =1$.
Note that $S^{n-1}$ is the $n$-dimensional unit sphere centered at the origin, $B_{\epsilon}^{\nu} $ is an $(n-1)$-dimensional $\epsilon$-ball which is centered at the origin and orthogonal to a unit vector $\nu$,  $\mathcal{L}^{n-1} $ is the $(n-1)$-dimensional Lebesgue measure,
$\delta$ is a function to be defined in the next section.
Moreover, $$\kint_{A } f(h) d \mathcal{L}^{n-1}(h): = \frac{1}{\mathcal{L}^{n-1}(A)}\int_{A } f(h) d \mathcal{L}^{n-1}(h) $$
for any $\mathcal{L}^{n-1}$-measurable functions $f$.
In Section 3, we verify that our value functions actually satisfy \eqref{dppvar}.
Once we know the relation between game values and \eqref{dppvar}, 
then we can concentrate on investigating the properties of functions satisfying this DPP.
By the Taylor expansion, we can expect that
if $u$ is a limit of game values $u_{\epsilon}$ as $\epsilon \to 0$, then it satisfies
$$ (n+p)u_{t} = \Delta_{p}^{N} u .$$
This observation provides motivation for the discussion in the last section.

The $p$-Laplace operator was first studied using tug-of-war games in \cite{MR2449057,MR2451291}.
Over the past decade, considerable progress has been made in the theory of value functions for tug-of-war games.
Several mean value characterizations for the $p$-Laplace operator are derived
in \cite{MR2566554,MR2684311,MR2875296}. 
Time-independent games have been studied in \cite{MR3011990,MR3169768,MR3441079,MR3623556,MR3846232,MR4125101,attouchi2021gradient}.
For time-dependent games, see also
\cite{MR2684311,MR3161604,MR3494400,MR3846232,MR4153524}.
We also refer the reader to \cite{MR2868849,MR2971208,MR3177660,MR3299035,MR3602849,lewicka2019robin,lewicka2019robin2} which deal with tug-of-war games under various settings. 
\\ \\
{\bf Acknowledgments}.
This work was supported by NRF-2019R1C1C1003844.
The author would like to thank M. Parviainen, for introducing this topic, valuable discussions and constant support throughout this work. 


\section{Preliminaries}
\subsection{Notations}

We still use the notation $B_{\epsilon}^{\nu} $ and $S^{n-1}$ as in Section 1.
Let $ \Omega $ be a bounded domain.
First we define 
$$ I_{\epsilon} = \{x \in \Omega  : \dist(x, \partial \Omega) < \epsilon \}  \ \textrm{and}$$
$$ O_{\epsilon} = \{x \in \mathbb{R}^{n} \backslash \overline{\Omega} : \dist(x, \partial \Omega) < \epsilon \}.   $$
We also set
$ \Gamma_{\epsilon}= I_{\epsilon}  \cup  O_{\epsilon} \cup \partial \Omega $ and 
$\Omega_{\epsilon} = \overline{\Omega} \cup  O_{\epsilon} .$

For $T>0$, consider a parabolic cylinder $ \Omega_{T} = \Omega \times (0,T] $ with its parabolic boundary $ \partial_{p} \Omega_{T} = ( \overline{\Omega} \times \{ 0 \} ) \cup (\partial \Omega \times (0,T])$.
Similarly to the elliptic case, we define parabolic $\epsilon$-strips
$$ I_{\epsilon, T} = \{ (x,t) \in \Omega  \times [ \epsilon^{2}/2, T] : \dist(x, \partial \Omega) < \epsilon \} \cup \big( \Omega \times (0, \epsilon^{2}/2) \big), $$ 
$$ O_{\epsilon, T} = \{(x,t) \in (\mathbb{R}^{n} \backslash \overline{\Omega}) \times (0,T] : \dist(x, \partial \Omega) < \epsilon \} \cup \big( \Omega_{\epsilon} \times (-\epsilon^{2}/2,0) \big)  $$ and
$$ \Gamma_{\epsilon,T} = I_{\epsilon, T}  \cup O_{\epsilon, T} \cup \partial_{p} \Omega_{T}.  $$
We denote  by $ \Omega_{\epsilon,T} $ the set $ \overline{\Omega}_{T} \cup O_{\epsilon,T} $.

For $ (x,t) \in \Omega_{T}$, we set a ``regularizing function'' $\delta$ in \eqref{dppvar} as follows:
\begin{align*} 
\delta(x,t) = 
\left\{ \begin{array}{ll}
0 & \textrm{in $\Omega_{T} \backslash  I_{\epsilon, T}, $}\\
\min \bigg\{ 1, 1 - \frac{\dist(x, \partial \Omega)}{\epsilon} \bigg\} \times \min \bigg\{  1,1 - \frac{\sqrt{2t}}{\epsilon}\bigg\} & \textrm{in $ I_{\epsilon, T}$, } \\
1 & \textrm{in $O_{\epsilon, T}$.}\\
\end{array} \right. 
\end{align*}
It is not difficult to check that $\delta$ is continuous in $\Omega_{\epsilon, T}$.

We introduce some notations for convenience. 
First, we write
$$ \midrg_{i \in I}A_{i}= \frac{1}{2} \bigg( \sup_{i \in I}A_{i} +\inf_{i \in I}A_{i} \bigg) .$$
And, we also denote by 
$$ \mathscr{A}_{\epsilon}u (x, t; \nu) = \alpha u(x+ \epsilon \nu,t) + \beta \kint_{B_{\epsilon}^{\nu} }u(x+ h,t) d \mathcal{L}^{n-1}(h)  $$
for a bounded measurable function $u$.
Then \eqref{dppvar} can be written as
\begin{align*}
& u_{\epsilon}(x,t)  = (1- \delta(x,t))\midrg_{\nu \in S^{n-1}}  \mathscr{A}_{\epsilon} u_{\epsilon} \bigg( x, t-\frac{\epsilon^{2}}{2}; \nu \bigg)   
 + \delta(x,t) F(x, t).
\end{align*} 
We call this bounded and measurable function $u$ a \emph{solution} to the DPP \eqref{dppvar}.
\subsection{Basic concepts} 

Let $\Omega \subset \mathbb{R}^{n}$ be a bounded domain, $T>0$ and  $\alpha, \beta \in (0,1)$ be fixed numbers  with $\alpha + \beta = 1$.
We also consider a function $F \in C(\Gamma_{\epsilon, T} )$.
From now on, we will use the symbol $u_{\epsilon}$ to denote a function satisfying the DPP
\eqref{dppvar} in $\Omega_{T}$ for given $F$. 

We consider two-player tug-of-war games related to \eqref{dppvar}.
There are various settings for these games. 
In particular, our setting can be regarded as a parabolic version of games in \cite{MR3623556}.

Our game setting is as follows.
There is a token located at a point $(x_{0},t_{0}) \in \Omega_{T}$.
Players will move it at each turn according to the outcome of the following processes.
We write locations of the token as $(x_{1},t_{1}),(x_{2},t_{2}), \cdots$ 
and denote by $Z_{j} =(x_{j},t_{j}) $ for our convenience.

When $Z_{j} \in \Omega \backslash I_{\epsilon} $, Player I and II choose some vectors
$ \nu_{j}^{\mathrm{I}}, \nu_{j}^{\mathrm{II}} \in  \partial B_{\epsilon}$. 
First, players compete to move token with a fair coin toss.
Next, they have one more stochastic process to determine how to move the token. 
The winner of first coin toss, Player $i\in \{ \mathrm{I}, \mathrm{II} \}$ can move the token to direction of the chosen vector $\nu_{j}^{i}$ with probability $ \alpha $. 
Otherwise, the token is moved uniformly random in the $(n-1)$-ball perpendicular to $\nu_{j}^{i}$.
After these processes are finished, $t_{j} $ is changed by $t_{j+1}= t_{j} - \epsilon^{2}/2 $.

If $Z_{j} \in \Gamma_{\epsilon} $, the game progresses in the same way as above with probability $ 1 - \delta (Z_{j})$. 
On the other hand, with probability $ \delta (Z_{j}) $, the game is over and Player II pays Player I  payoff $F(Z_{j})$.

We denote by $\tau$ the number of total turns until end of the game.
Observe that $\tau$ must be finite in our setting since the game ends when $ t \le 0$.

Now we give mathematical construction for this game.
Let $ \xi_{0}, \xi_{1}, \cdots $ be iid random variables to have a uniform distribution $U(0,1) $.
This process $\{\xi_{j} \}_{j=0}^{\infty} $ is independent of $ \{ Z_{j} \}_{j=0}^{\infty} $.

Define $\tilde{C}:=\{ 0,1 \} $.
We set random variables $c_{0}, c_{1}, \cdots \in \tilde{C} $ as follows: 
\begin{align*} 
c_{j} = 
\left\{ \begin{array}{ll}
0 & \textrm{when $\xi_{j-1} \le 1- \delta(Z_{j-1}) $,}\\
1 & \textrm{when $\xi_{j-1} > 1- \delta(Z_{j-1})$}\\
\end{array} \right. 
\end{align*}
for $j \ge 1$ and $c_{0}=0$.
Then we can write the stopping time $\tau$ by $$\tau := \inf \{  j \ge 0 : c_{j+1}=1 \} .$$

In our game, each player chooses their strategies by using past data (history).
We can write a history as the following vector 
$$ \big((c_{0},Z_{0}),(c_{1},Z_{1}),\cdots,(c_{j},Z_{j})\big).$$
Then, the strategy of Player $i$ can be defined by a Borel measurable function as
$\mathcal{S}_{i} = \{S_{i}^{j}\}_{j=1}^{\infty}$ with
$$ S_{i}^{j}: \{ (c_{0}, Z_{0}) \} \times \cup_{k=1}^{j-1}(\tilde{C} \times \Omega_{\epsilon,T}) \to \partial B_{\epsilon}(0)$$
for any $j \in \mathbb{N} $.

Now we define a probability measure $\mathbb{P}_{S_{\mathrm{I}},S_{\mathrm{II}}}^{Z_{0}}$ natural product $\sigma$-algebra of the space of all game trajectories for any starting point $Z_{0} \in \Omega_{\epsilon,T}  $.
By Kolmogorov's extension theorem, we can construct the measure to the family of transition densities
\begin{align*}
&\pi_{S_{\mathrm{I}},S_{\mathrm{II}}}((c_{0},Z_{0}),(c_{1},Z_{1}),\cdots,(c_{j},Z_{j}); C,A_{j+1}) \\ &
= (1 - \delta(Z_{j})) \pi_{S_{\mathrm{I}},S_{\mathrm{II}}}^{local}((Z_{0},Z_{1},\cdots,Z_{j});A_{j+1}) \mathbb{I}_{0}(C) \mathbb{I}_{c_{j}}(\{ 0 \}) \\ &
\quad + \delta(Z_{j})\mathbb{I}_{Z_{j}}(A_{j})\mathbb{I}_{1}(C) \mathbb{I}_{c_{j}}(\{ 0 \})
+\mathbb{I}_{Z_{j}}(A_{j}) \mathbb{I}_{c_{j}}(\{ 1 \})
\end{align*}
for $A_{n} = A \times \{ t_{n} \} $ ($A$ is any Borel set in $\mathbb{R}^{n}$ and $n \ge 0 $) and $C \subset \tilde{C}$,
where
\begin{align*}
& \pi_{S_{\mathrm{I}},S_{\mathrm{II}}}^{local}(Z_{0},Z_{1},\cdots,Z_{j};A_{j+1}) \\ &
= \frac{1}{2} \bigg[ \alpha ( \mathbb{I}_{(x_{j}+\nu_{j+1}^{I}, t_{j+1})}(A_{j+1}) + \mathbb{\mathrm{I}}_{(x_{j}+\nu_{j+1}^{\mathrm{II}}, t_{j+1})}(A_{j+1}) \big)
\\ & \qquad + \frac{\beta}{\omega_{n-1} \epsilon^{n-1}} \big(
\mathcal{L}^{n-1}(B_{\epsilon}^{\nu_{j+1}^{\mathrm{I}}} (Z_{j}) \cap A_{j+1})+\mathcal{L}^{n-1}(B_{\epsilon}^{\nu_{j+1}^{\mathrm{II}}}(Z_{j}) \cap A_{j+1}
\big) \bigg].
\end{align*}
Here, $\omega_{n-1}= \mathcal{L}^{n-1}(B_{1}^{n-1})$ where $ B_{1}^{n-1}$ is the $(n-1)$-dimensional unit ball and 
\begin{align*} 
\mathbb{I}_{z}(B) = 
\left\{ \begin{array}{ll}
0 & \textrm{when $z \notin B $,}\\
1 & \textrm{when $z \in B$.}\\
\end{array} \right. 
\end{align*}

Finally, for any starting point $Z_{0}=(x_{0},t_{0}) \in \Omega_{T} $, we define value functions $u_{\mathrm{I}} $ and $u_{\mathrm{II}}$ of this game for Player I and II by
$$u_{\mathrm{I}}(Z_{0})  =  \sup_{S_{\mathrm{I}}} \inf_{S_{\mathrm{II}}} \mathbb{E}_{S_{\mathrm{I}},S_{\mathrm{II}}}^{Z_{0}}[F(Z_{\tau})]$$ and $$u_{\mathrm{II}}(Z_{0})  = \inf_{S_{\mathrm{II}}} \sup_{S_{\mathrm{I}}} \mathbb{E}_{S_{\mathrm{I}},S_{\mathrm{II}}}^{Z_{0}}[F(Z_{\tau})],$$
respectively.

\section{The existence and uniqueness of game values}

In this section, we study the existence and uniqueness of functions satisfying the DPP \eqref{dppvar} with continuous boundary data $F$.
Moreover, we also observe the relation between these functions and value functions for time-dependent tug-of-war games.

Before showing the existence and uniqueness of functions satisfying \eqref{dppvar}, we need to check a subtle issue.
In the DPP, the value of $u_{\epsilon}(x,t)$ is determined by values of the function in $B_{\epsilon}(x) \times \{ t - \epsilon^{2}/2 \} $.
And we also see that \eqref{dppvar} contains integral terms for the function at time $t - \epsilon^{2}/2   $.
Thus, we have to consider the measurability for the function $u_{\epsilon}$, more precisely, for strategies of our game.

In general, existence of measurable strategies is not guaranteed (for example, see \cite[Example 2.4]{MR3161602}). 
But we can avoid this problem under our setting.
The ``regularizing function'' $\delta$ plays an important role in this issue.

We begin this section by observing a basic property of the operator $ \mathscr{A}_{\epsilon}$. 
\begin{proposition} \label{contilem} Let $u \in  C ( \overline{\Omega}_{\epsilon,T} )$. 
Then  $  \mathscr{A}_{\epsilon}u (x, t; \nu) $ is continuous with respect to each variable in $\overline{\Omega}_{T, \epsilon} \times \partial B_{\epsilon}(0) $.
\end{proposition}
\begin{proof}
For any $(x,t),(y,s) \in \Omega_{T} $, let us define a parabolic distance by $d( (x,t), (y,s)) = | x-y | + |t-s|^{1/2}$.
We write the modulus of continuity of a function $f$ with respect to the distance $d$   by $\omega_{f} $.

For fixed $|\nu| = \epsilon $, we can see that for any $ x, y \in \overline{\Omega}$,
\begin{align*}
\bigg| \alpha u \bigg(x+ \epsilon \nu, t-\frac{\epsilon^{2}}{2} \bigg) - \alpha u \bigg(y+ \epsilon \nu, t-\frac{\epsilon^{2}}{2} \bigg)\bigg| \le \alpha \omega_{u}(|x-y|)
\end{align*}
and
\begin{align*}
\bigg| \beta & \kint_{B_{\epsilon}^{\nu} }  u \bigg(\hspace{-0.15em}x+h,t-\frac{\epsilon^2}{2} \bigg) d \mathcal{L}^{n-1}(h) - \beta \kint_{B_{\epsilon}^{\nu} }  u \bigg(\hspace{-0.15em}y+h,t-\frac{\epsilon^2}{2} \bigg) d \mathcal{L}^{n-1}(h) \bigg| \\ &
\le \beta \kint_{B_{\epsilon}^{\nu} } \bigg|  u \bigg(\hspace{-0.15em}x+h,t-\frac{\epsilon^2}{2} \bigg)  -  u \bigg(\hspace{-0.15em}y+h,t-\frac{\epsilon^2}{2} \bigg)  \bigg|d \mathcal{L}^{n-1}(h) \\ &
\le \beta \omega_{u} (|x-y|).
\end{align*}
Thus, we get
$$ |  \mathscr{A}_{\epsilon}u (x, t;\nu)  - \mathscr{A}_{\epsilon}u  (y, t;\nu) | \le \omega_{u}(|x-y|).$$
Next, for any $ t,s >0$, we also calculate that
\begin{align*}
\bigg| \alpha u \bigg(x+ \epsilon \nu, t-\frac{\epsilon^{2}}{2} \bigg) - \alpha u \bigg(x+ \epsilon \nu, s-\frac{\epsilon^{2}}{2} \bigg)\bigg| \le \alpha \omega_{u}(|t-s|^{1/2}),
\end{align*}
\begin{align*}
\bigg| \beta & \kint_{B_{\epsilon}^{\nu} }  u \bigg(\hspace{-0.15em}x+h,t-\frac{\epsilon^2}{2} \bigg) d \mathcal{L}^{n-1}(h) - \beta \kint_{B_{\epsilon}^{\nu} }  u \bigg(\hspace{-0.15em}x+h,s-\frac{\epsilon^2}{2} \bigg) d \mathcal{L}^{n-1}(h) \bigg| \\ &
\le \beta \kint_{B_{\epsilon}^{\nu} } \bigg|  u \bigg(\hspace{-0.15em}x+h,t-\frac{\epsilon^2}{2} \bigg)  -  u  \bigg(\hspace{-0.15em}x+h,s-\frac{\epsilon^2}{2} \bigg)  \bigg|d \mathcal{L}^{n-1}(h) \\ &
\le \beta \omega_{u} (|t-s|^{1/2})
\end{align*}
and hence
$$ |  \mathscr{A}_{\epsilon}u (x, t; \nu)  - \mathscr{A}_{\epsilon}u (x, s; \nu) | \le \omega_{u}(|t-s|^{1/2}).$$
Finally, for any $ \nu, \chi \in S^{n-1}$,
\begin{align*}
&W(x,t,\nu)-W(x,t, \chi) \\&
 =\alpha \bigg[  u \bigg(x+ \epsilon \nu, t-\frac{\epsilon^{2}}{2} \bigg) - u \bigg(x+ \epsilon \chi, t-\frac{\epsilon^{2}}{2} \bigg) \bigg] \\&
+ \beta \bigg[ \kint_{B_{\epsilon}^{\nu} }  u  \bigg(\hspace{-0.15em}x+h,t-\frac{\epsilon^2}{2} \bigg) d \mathcal{L}^{n-1}(h) - \kint_{B_{\epsilon}^{\chi} }  u  \bigg(\hspace{-0.15em}x+h,t-\frac{\epsilon^2}{2} \bigg) d \mathcal{L}^{n-1}(h) \bigg].
\end{align*}
Combining the above results, we see that
\begin{align*}
& |W(x,t,\nu)-W(x,t, \chi)  | \\ &
\le \alpha \epsilon \omega_{u} (\epsilon|\nu-\chi|) \\& +  \beta \kint_{B_{\epsilon}^{\nu} } \bigg|  u \bigg(\hspace{-0.15em}x+h,t-\frac{\epsilon^2}{2} \bigg)  - u \bigg(\hspace{-0.15em}x+Ph,t-\frac{\epsilon^2}{2} \bigg)  \bigg|d \mathcal{L}^{n-1}(h) 
\end{align*}
where $P: \nu^{\perp} \to \chi^{\perp}$ is a rotation satisfying $ |h-Ph| \le C|h||\nu-\chi|$.
Here, we check that
\begin{align*}
\kint_{B_{\epsilon}^{\nu} } \bigg|  u  \bigg(\hspace{-0.15em}x+h,t-\frac{\epsilon^2}{2} \bigg)  -  u \bigg(\hspace{-0.15em}x+Ph,t-\frac{\epsilon^2}{2} \bigg)  \bigg|d \mathcal{L}^{n-1}(h) & \le \omega_{u}(|h-Ph|) 
\\ &  \le \omega_{u}(C \epsilon|\nu-\chi|).
\end{align*}
Therefore, we obtain $$|  \mathscr{A}_{\epsilon}u_{\epsilon} (x, t;\nu)  - \mathscr{A}_{\epsilon}u_{\epsilon} (x, t;\chi) | \le \omega_{u}(C\epsilon|\nu-\chi|). $$
Now we can conclude the proof to combine above results.
\end{proof}

Next we observe that the operator $T$ preserves continuity and monotonicity.
For convenience, we write that
\begin{align} \label{deft}
&Tu(x,t) = (1- \delta(x,t))\midrg_{\nu \in S^{n-1}} \mathscr{A}_{\epsilon}u \bigg( x, t- \frac{\epsilon^{2}}{2} ; \nu \bigg)
 + \delta(x,t) F(x, t ).
\end{align}


\begin{lemma} \label{mnpc}
For any $ u \in C(\overline{\Omega}_{\epsilon,T})$, $ Tu$ is also in $ C(\overline{\Omega}_{\epsilon,T}) $. 
Furthermore, for any  $ u, v \in C(\overline{\Omega}_{\epsilon,T})$ with $ u \le v$, it holds that
$$ Tu \le Tv.$$
\end{lemma}
\begin{proof}
By the definition of $T$, we can check that $u \le v $ implies  $ Tu \le Tv$ without difficulty.

Next we need to show that $Tu \in C(\overline{\Omega}_{\epsilon,T}) $ if $ u \in C(\overline{\Omega}_{\epsilon,T})$. 
When $ (x,t) \in \overline{O}_{\epsilon,T} $, we see that $Tu=u=F \in C(\overline{O}_{\epsilon,T}) $ by assumption.
We need to consider the case of $ \overline{I}_{\epsilon,T}$ and $\Omega_{T} \backslash I_{\epsilon,T} $.

First assume that $ (x,t), (y,s) \in \Omega_{T} \backslash I_{\epsilon,T}  $. 
Observe that
\begin{align*}
& \big| \midrg_{\nu \in S^{n-1}} \mathscr{A}_{\epsilon}u (x,t;\nu) -\midrg_{\nu \in S^{n-1}} \mathscr{A}_{\epsilon}u (y,s; \nu) \big| \\ &
\le \frac{1}{2} \big| \sup_{\nu \in S^{n-1}} \mathscr{A}_{\epsilon}u (x,t; \nu) -\sup_{\nu \in S^{n-1}} \mathscr{A}_{\epsilon}u (y,s; \nu) \big| \\ & \qquad + 
\frac{1}{2} \big| \inf_{\nu \in S^{n-1}} \mathscr{A}_{\epsilon}u (x,t;\nu) -\inf_{\nu \in S^{n-1}} \mathscr{A}_{\epsilon}u (y,s; \nu) \big| .
\end{align*}
Since
$$  \big| \sup_{\nu \in S^{n-1}} \mathscr{A}_{\epsilon}u (x, t;\nu) -\sup_{\nu \in S^{n-1}} \mathscr{A}_{\epsilon}u (y, s; \nu) \big| \le \sup_{\nu \in S^{n-1}} | \mathscr{A}_{\epsilon}u (x, t; \nu) - \mathscr{A}_{\epsilon}u (y, s;\nu)  | $$ and
$$ \big| \inf_{\nu \in S^{n-1}} \mathscr{A}_{\epsilon}u (x,t; \nu) -\inf_{\nu \in S^{n-1}} \mathscr{A}_{\epsilon}u (y, s; \nu) \big| \le \sup_{\nu \in S^{n-1}} | \mathscr{A}_{\epsilon}u (x, t; \nu) - \mathscr{A}_{\epsilon}u (y,s; \nu)  |,$$
we get
\begin{align*}
 \big| \midrg_{\nu \in S^{n-1}} &\mathscr{A}_{\epsilon}u (x, t; \nu) -\midrg_{\nu \in S^{n-1}} \mathscr{A}_{\epsilon}u (y, s; \nu) \big| \\ &
\le  \sup_{\nu \in S^{n-1}} | \mathscr{A}_{\epsilon}u (x,t; \nu) - \mathscr{A}_{\epsilon}u (y,s; \nu)  | \\ &
\le  \omega_{u}(d((x,t),(y,s))).
\end{align*}
We used the result of Proposition \ref{contilem} in the last inequality.
Thus, $ Tu$ is also continuous in $\Omega_{T} \backslash I_{\epsilon,T} $.

When $ (x,t), (y,s) \in \overline{I}_{\epsilon,T}$,
\begin{align*}
 \big| (1-& \delta(x,t))\midrg_{\nu \in S^{n-1}} \mathscr{A}_{\epsilon}u ( x, t; \nu)
- (1- \delta(y,s))\midrg_{\nu \in S^{n-1}} \mathscr{A}_{\epsilon}u ( y,s; \nu) \big|
\\ & \le (1- \delta(x,t)) \big| \midrg_{\nu \in S^{n-1}} \mathscr{A}_{\epsilon}u (x,t; \nu) -\midrg_{\nu \in S^{n-1}} \mathscr{A}_{\epsilon}u (y,s; \nu) \big| \\ & \qquad
+ |\delta(x,t)-\delta(y,s) \big|\cdot \big| \midrg_{\nu \in S^{n-1}} \mathscr{A}_{\epsilon}u (y,s; \nu) \big| \\ &
\le \omega_{u}(d((x,t),(y,s))) + \frac{3}{\epsilon}||u||_{\infty} d((x,t),(y,s))
\end{align*}
because
\begin{align*}
|\delta(x,t)-\delta(y,s) | \le \frac{3}{\epsilon} d((x,t),(y,s)) .
\end{align*}
Similarly, we can also calculate
\begin{align*}
&|\delta(x,t)F(x,t)-\delta(y,s)F(y,s) | \\ & \le
 \omega_{F}(d((x,t),(y,s)) )+ \frac{3}{\epsilon}||F||_{\infty} d((x,t),(y,s)).
\end{align*}
Combining above results, we obtain the continuity of $Tu $ in $ I_{\epsilon,T}$. 

Finally, we need to check the coincidence of the function value on $\partial I_{\epsilon,T}$.
Observe that $\partial I_{\epsilon,T}$ can be decomposed by two disjoint connected sets $\partial_{p} \Omega_{T}  $ and $ \partial_{p}( \Omega_{T}\backslash I_{\epsilon,T}) $ of $\mathbb{R}^{n} \times \mathbb{R}$.
Then we can observe that 
$$  \lim_{O_{\epsilon,T} \ni (y,s) \to (x,t)} Tu(y,s) = \lim_{I_{\epsilon,T} \ni (y,s) \to (x,t)} Tu(y,s) = Tu(x,t)$$
for any $(x,t) \in \partial_{p} \Omega_{T} $ and
$$  \lim_{\Omega_{T}\backslash I_{\epsilon,T} \ni (y,s) \to (x,t)} Tu(y,s) = \lim_{I_{\epsilon,T} \ni (y,s) \to (x,t)} Tu(y,s) = Tu(x,t)$$
for any $(x,t) \in \partial_{p}( \Omega_{T}\backslash I_{\epsilon,T}) $
by using the above calculation.
Thus we obtain the continuity of $Tu$ and the proof is finished.
\end{proof}
Since $T$ preserves continuity, we do not need to worry about the measurability issue.
Therefore, for any continuous function $u$, $T u$ is well-defined at every point in $\Omega_{T}$.

Now we can obtain the existence and uniqueness of these functions.
\begin{theorem}
Let $F \in C(\Gamma_{\epsilon,T} )$.
Then the bounded function $ u_{\epsilon}$ satisfying \eqref{dppvar} with boundary data $F$ exists and is unique.
\end{theorem}
\begin{proof}
We get the desired result via an argument similar to the proof of \cite[Theorem 5.2]{MR3161602}.
We can see the existence of these functions without difficulty since the operator $ T$ is well-defined inductively for any continuous boundary data.

For uniqueness, consider two functions $ u$ and $v$ satisfying $Tu=u $, $Tv=v $ with boundary data $F$.
We see that $ u(\cdot, t)= v(\cdot, t)$ when $0< t \le \epsilon^{2}/2 $ by definition of $T $.
Then we can also get the same result when $\epsilon^{2}/2  < t \le \epsilon^{2} $ because past data of $u $ and $v$ still coincide.
Repeating this process, we obtain $ u(x,t)=v(x,t)$ for any $(x,t) \in \Omega_{T} $ and hence the uniqueness is proved.
\end{proof}

We look into the relation between functions satisfying \eqref{dppvar} and values for parabolic tug-of-war games here.
\begin{theorem}
The value functions of tug-of-war game with noise $u_{\mathrm{I}} $ and $ u_{\mathrm{II}}$ with payoff function $F$ coincide with the function $u_{\epsilon} $.
\end{theorem}
\begin{proof}
We need to deduce that
\begin{align*}
u_{\epsilon} \le u_{\mathrm{I}} \qquad \textrm{and} \qquad  u_{\mathrm{II}} \le u_{\epsilon} 
\end{align*}
since $u_{\mathrm{I}} \le u_{\mathrm{II}} $ by the definition of value functions.

First, we show the latter inequality.
Let $Z_{0} \in \Omega_{T} $ and denote by $S_{\mathrm{II}}^{0} $ a strategy  for Player {II} such that 
$$\mathscr{A}_{\epsilon}u_{\epsilon}(Z_{j}; \nu_{j}^{\mathrm{II}}) = \inf_{\nu \in S^{n-1}}\mathscr{A}_{\epsilon}u_{\epsilon}(Z_{j};\nu)$$
for $ j \ge 0$.
Note that this $ S_{\mathrm{II}}^{0}$ exists since $\mathscr{A}_{\epsilon}u_{\epsilon} $ is continuous on $\nu$ by Proposition \ref{contilem}.
Measurability of such strategies can be shown by using \cite[Theorem 5.3.1]{MR1619545}.

Next we fix an arbitrary strategy $S_{I} $ for Player $I$. 
Define
\begin{align*} 
\Phi(c,x,t)= 
\left\{ \begin{array}{ll}
u_{\epsilon}(x,t)& \textrm{when $c=0 $,}\\
F(x,t) & \textrm{when $c=1$.}\\
\end{array} \right. 
\end{align*}
for any $(x,t) \in \Omega_{\epsilon,T}$.
Then we have
\begin{align*}
&\mathbb{E}_{S_{\mathrm{I}},S_{\mathrm{II}}^{0}}^{Z_{0}}[\Phi(c_{j+1},Z_{j+1})|(c_{0},Z_{0}),\cdots,(c_{j},Z_{j})] \\ & \le \frac{1-\delta(Z_{j})}{2} \big[ \mathscr{A}_{\epsilon}u_{\epsilon}(x_{j}, t_{j+1}; \nu_{j+1}^{\mathrm{I}}) +  \mathscr{A}_{\epsilon}u_{\epsilon}(x_{j}, t_{j+1}; \nu_{j+1}^{\mathrm{II}})  \big]+ \delta(Z_{j})F(Z_{j})
\\ & \le (1-\delta(Z_{j})) \midrg_{\nu \in S^{n-1}} \mathscr{A}_{\epsilon}u_{\epsilon} (x_{j},t_{j+1}; \nu) + \delta(Z_{j})F(Z_{j}) \\ &
= \Phi(c_{j},Z_{j}) .
\end{align*}
Hence, we can see that $M_{k}= \Phi(c_{k},Z_{k})$ is a supermartingale in this case.
Since the game ends in finite steps, we can obtain 
\begin{align*}
u_{\mathrm{II}}(Z_{0}) & = \inf_{S_{\mathrm{II}}} \sup_{S_{\mathrm{I}}} \mathbb{E}_{S_{\mathrm{I}},S_{\mathrm{II}}}^{Z_{0}}[F(Z_{\tau})] 
\le  \sup_{S_{\mathrm{I}}} \mathbb{E}_{S_{\mathrm{I}},S_{\mathrm{II}}^{0}}^{Z_{0}}[F(Z_{\tau})] \\ &
=  \sup_{S_{\mathrm{I}}} \mathbb{E}_{S_{\mathrm{I}},S_{\mathrm{II}}^{0}}^{Z_{0}}[\Phi(c_{\tau+1},Z_{\tau+1})] 
  \le \sup_{S_{\mathrm{I}}} \mathbb{E}_{S_{\mathrm{I}},S_{\mathrm{II}}^{0}}^{Z_{0}}[\Phi(c_{0},Z_{0})] \\ &
= u_{\epsilon}(Z_{0})
\end{align*}
by using optional stopping theorem.

Now we can also derive that $u_{\epsilon} \le u_{\mathrm{I}}$ by using a similar argument.
Then we get the desired result.
\end{proof}

\section{Long-time asymptotics}
In PDE theory, the study of asymptotic behavior of solutions of parabolic equations as time goes to infinity 
has drawn a lot of attention.
We will have a similar discussion for our value function $u_{\epsilon}$ when the boundary data $F$ 
does not depend on $t$ in $\Gamma_{\epsilon} \times (\epsilon^{2}, \infty)$.
The heuristic idea in this section can be summarized as follows. 
Assume that we start the game at $(x_{0},t_{0})$ for sufficiently large $t_{0}$. 
Then we can expect that the probability of the game ending in the initial boundary would be close to zero,
that is, the game finishes on the lateral boundary in most cases.  
Since we assumed that $F$ is independent of $t$ for $t > \epsilon^{2}$, 
we may consider this game as something like a time-independent game with the same boundary data. 
Thus, it is reasonable to guess that the value function of the time-dependent game converges that of the corresponding time-independent game.
We refer the reader to \cite{ber2019evolution} 
which contains a detailed discussion of asymptotic behaviors for value functions of evolution problems.
Moreover, long-time asymptotics for related PDEs can be found in
\cite{MR648452,MR1977429,MR2915863,pv2019equivalence}.

To observe the asymptotic behavior of value functions, we first need to obtain the following comparison principle. 
Since it can be shown in a straightforward manner by using the DPP \eqref{dppvar}, we omit the proof.
One can find similar results  in \cite[Theorem 5.3]{MR3161602}.

\begin{lemma} \label{mono}
Let $u$ and $v$ be functions satisfying \eqref{dppvar} with boundary data $F_{u}$ and $ F_{v} $, respectively.
Suppose that $ F_{u} \le F_{v}$ in $\Gamma_{\epsilon,T}$. 
Then,
$$ u \le v \qquad \textrm{in} \ \Omega_{\epsilon,T}. $$
\end{lemma}

Now we state the main result of this section.
\begin{theorem} \label{stable}
Let $ \Omega$ be a bounded domain.
Consider functions  $\psi \in C(\Gamma_{\epsilon})$ and  $\varphi \in C(\Gamma_{\epsilon, T} \cap \{ t \le 0 \})$,
and define a function $F  \in C(\Omega_{\epsilon, T})$ as follows:
\begin{align} \label{stabof} F(x,t)=
\left\{ \begin{array}{ll}
\psi(x) & \textrm{ in $ \Gamma_{\epsilon} \times (\epsilon^{2},T]$,}\\
\varphi(x,\epsilon^{2}/2) + \frac{2t(\psi(x)-\varphi(x,\epsilon^{2}/2))}{\epsilon^{2}} & \textrm{ in $ \Gamma_{\epsilon} \times (\frac{\epsilon^{2}}{2} ,\epsilon^{2}] $,}\\
\varphi(x,t) & \textrm{in $ \Omega_{\epsilon} \times [-\frac{\epsilon^{2}}{2},  \frac{\epsilon^{2}}{2}]$.}\\
\end{array} \right. 
\end{align}
Assume that $u_{\epsilon}$ is the function satisfying \eqref{dppvar} with boundary data $F$.
Then we have
$$ \lim_{T \to \infty} u_{\epsilon} (x,T) = U_{\epsilon}(x)$$
where $ U_{\epsilon}$ is the function satisfying the following DPP
\begin{align} \begin{split} \label{dpplim}
& U_{\epsilon} (x) \\ & = (1- \overline{\delta}(x))\midrg_{\nu \in S^{n-1}} \bigg[ \alpha U_{\epsilon} ( x+ \epsilon \nu ) + \beta \kint_{B_{\epsilon}^{\nu} } U_{\epsilon} (\hspace{-0.15em}x+h) d \mathcal{L}^{n-1}(h) \hspace{-0.15em}  \bigg] 
\\ & \qquad  + \overline{\delta}(x) \psi(x)
 \end{split}
\end{align}
in $\Omega_{\epsilon} $
with boundary data $\psi $
where 
\begin{align*} \overline{\delta}(x): = \lim_{t \to \infty} \delta (x,t) 
=\left\{ \begin{array}{ll}
0 & \textrm{in $\Omega  \backslash  I_{\epsilon }, $}\\
 1- \dist(x,\partial \Omega)/\epsilon & \textrm{in $ I_{\epsilon}$, } \\
1 & \textrm{in $O_{\epsilon}$.}\\
\end{array} \right. 
\end{align*}
\end{theorem}

\begin{remark}
We can find the existence and uniqueness of value functions under different setting in \cite{MR3161602},
which is related to the normalized $p$-Laplace operator for $p \ge 2$.
In that paper, the existence of measurable strategies is shown without regularization.
Thus, we do not have to consider a ``regularized function'' such as $\varphi(x,\epsilon^{2}/2) + 2t(\psi(x)-\varphi(x,\epsilon^{2}/2))/\epsilon^{2} $ in that case.
Meanwhile, for the time-independent version of our settings, results for these issues are shown in \cite{MR3471974}.
\end{remark}

\begin{proof}[Proof of Theorem \ref{stable}]
We will set some proper barrier functions $\underline{u}, \overline{u} $ such that
$$ \underline{u} \le  u_{\epsilon} \le \overline{u}$$
and show the coincidence for the limits of two barrier functions as $ t \to \infty$. 
In our proof, the uniqueness result for elliptic games is essential.
The motivation of this proof is from \cite[Proposition 3.3]{MR3623556}.

Let $\underline{\varphi} , \overline{\varphi}  $ be constants defined by 
$$ \underline{\varphi} = \min  \{ \inf_{\Gamma_{\epsilon}} \psi , \inf_{\Omega_{\epsilon}}\varphi \} \  \textrm{and} \ \overline{\varphi} =  \max \{ \sup_{\Gamma_{\epsilon}} \psi , \sup_{\Omega_{\epsilon}}\varphi \},$$
respectively. 
We consider $ \underline{u}$, $\overline{u} $ be functions satisfying \eqref{dppvar} with boundary data $\underline{F} $ and $\overline{F} $, where
\begin{align*} \underline{F}(x,t)=
\left\{ \begin{array}{ll}
\psi(x) & \textrm{  in $ \Gamma_{\epsilon} \times (\epsilon^{2},T]$ ,}\\
\underline{\varphi}  + 2t(\psi(x)-\underline{\varphi} )/\epsilon^{2} & \textrm{ in $ \Gamma_{\epsilon} \times (\frac{\epsilon^{2}}{2} ,\epsilon^{2}] $,}\\
\underline{\varphi}  & \textrm{in $ \Omega_{\epsilon} \times [-\frac{\epsilon^{2}}{2},  \frac{\epsilon^{2}}{2}]$,}\\
\end{array} \right. 
\end{align*}
and
\begin{align*} \overline{F}(x,t)=
\left\{ \begin{array}{ll}
\psi(x) & \textrm{ in $ \Gamma_{\epsilon} \times (\epsilon^{2},T] $,}\\
\overline{\varphi} + 2t(\psi(x)-\overline{\varphi})/\epsilon^{2} & \textrm{ in $ \Gamma_{\epsilon} \times (\frac{\epsilon^{2}}{2} ,\epsilon^{2}] $,}\\
\overline{\varphi}  & \textrm{in $ \Omega_{\epsilon} \times [-\frac{\epsilon^{2}}{2},  \frac{\epsilon^{2}}{2}]$,}\\
\end{array} \right. 
\end{align*}
respectively. 
Note that $ \underline{F}$ and $ \overline{F}$ are continuous in $\overline{\Gamma_{\epsilon,T}}$ and have constant initial data.

By Lemma \ref{mono}, we have $\underline{u} \le u_{\epsilon} \le \overline{u} $.
Thus it is sufficient to show that $\lim_{t \to \infty} \underline{u}(\cdot,t), \lim_{t \to \infty}\overline{u}(\cdot,t)$ exist and satisfy the limiting DPP
\eqref{dpplim}.
First we see that $$ ||\underline{u}||_{L^{\infty}(\Omega_{\epsilon,T})} \le  ||\underline{F}||_{L^{\infty}(\Gamma_{\epsilon,T})}   \ \textrm{and} \ ||\overline{u}||_{L^{\infty}(\Omega_{\epsilon,T})} \le  ||\overline{F}||_{L^{\infty}(\Gamma_{\epsilon,T})}   $$
by using the DPP of $ \underline{u}$ and $\overline{u} $. 
Thus, these functions are uniformly bounded.

Next, we prove monotonicity of sequences $\{ \underline{u}(x, t+ j\epsilon^{2}/2 ) \}_{j = 0}^{\infty} $ and $\{ \overline{u}(x, t+ j\epsilon^{2}/2) \}_{j = 0}^{\infty} $ for any $(x,t) \in \Omega_{\epsilon} \times (-\epsilon^{2}/2 ,0]$.
Without loss of generality, we only consider the case $\underline{u}$.
Let $(x_{0},t_{0}) $ be a point in $\Omega \times ( -\epsilon^{2}/2, 0] $ and
denote by $$a_{j} =  \underline{u}(x_{0}, t_{0}+ j\epsilon^{2}/2 )$$ for simplicity.
For any $(x_{0},t_{0}) \in \Omega \times ( -\epsilon^{2}/2, 0] $, we can derive that $$\underline{\psi}=  a_{0} = a_{1} \le a_{2}  $$ by direct calculation 
and 
\begin{align*}
a_{3} & = \bigg(1- \delta  \bigg(x_{0},t_{0}+\frac{3 \epsilon^{2}}{2}\bigg)\bigg)\midrg_{\nu \in S^{n-1}} \mathscr{A}_{\epsilon}u  ( x_{0}, t_{0}+  \epsilon^{2} ; \nu  ) \\ &
\qquad + \delta  \bigg(x_{0},t_{0}+\frac{3 \epsilon^{2}}{2}\bigg) F  \bigg(x_{0},t_{0}+\frac{3 \epsilon^{2}}{2}\bigg) \\ &
\ge  (1- \delta  (x_{0},t_{0}+  \epsilon^{2} ) )\midrg_{\nu \in S^{n-1}} \mathscr{A}_{\epsilon}u  \bigg( x_{0}, t_{0}+  \frac{\epsilon^{2}}{2} ; \nu  \bigg) \\ &
\qquad + \delta   (x_{0},t_{0}+ \epsilon^{2} ) F   (x_{0},t_{0}+  \epsilon^{2} ) =a_{2}
\end{align*} 
since $ \delta (x_{0},t_{0}+  \epsilon^{2} )=\delta (x_{0},t_{0}+  3\epsilon^{2}/2 ) $ and $ F (x_{0},t_{0}+  \epsilon^{2} ) \le F (x_{0},t_{0}+  3\epsilon^{2}/2 ) $.

Next, assume that $a_{k} \ge a_{k-1} $ for some $ k \ge 4 $.
Note that $F(x,t) = \psi (x) $ for $  x \in \Gamma_{\epsilon} $ and 
$$\delta \bigg(x_{0},t_{0}+\frac{k\epsilon^{2}}{2}\bigg) = \delta \bigg(x_{0},t_{0}+\frac{(k-1)\epsilon^{2}}{2}\bigg) $$  in this case.
Then, we see
\begin{align*}
a_{k+1}  & = \bigg(1- \delta \bigg(x_{0},t_{0}+\frac{(k+1)\epsilon^{2}}{2}\bigg)\bigg)\midrg_{\nu \in S^{n-1}} \mathscr{A}_{\epsilon}u \bigg( x_{0}, t_{0}+ \frac{k\epsilon^{2}}{2}; \nu \bigg) \\ &
\qquad + \delta \bigg(x_{0},t_{0}+\frac{(k+1)\epsilon^{2}}{2} \bigg) \psi(x_{0}) \\ &
\ge \bigg(1- \delta \bigg(x_{0},t_{0}+\frac{k\epsilon^{2}}{2}\bigg)\bigg)\midrg_{\nu \in S^{n-1}} \mathscr{A}_{\epsilon}u \bigg( x_{0}, t_{0}+ \frac{(k-1)\epsilon^{2}}{2}; \nu \bigg) \\ &
\qquad + \delta \bigg(x_{0},t_{0}+\frac{k\epsilon^{2}}{2} \bigg) \psi(x_{0}) = a_{k}.
\end{align*}
Therefore, $\{ a_{j} \} $ is increasing for any $(x_{0},t_{0}) \in \Omega \times ( -\epsilon^{2}/2, 0]$. 
It is also possible to obtain $  \{ \overline{u}(x, t+ j\epsilon^{2}/2) \}$ is decreasing by using similar arguments.
Therefore, we obtain  $\{ \underline{u}(x, t+ j\epsilon^{2}/2 ) \}  $ and $\{ \overline{u}(x, t+ j\epsilon^{2}/2) \}  $ converges for any $(x,t) \in \Omega_{\epsilon} \times (-\epsilon^{2}/2 ,0]$ by applying the monotone convergence theorem.

Now we show that $U_{\epsilon}$ satisfies the DPP \eqref{dpplim}.
Fix $- \epsilon^{2}/2 \le t_{1} < 0 $ arbitrary and
write $$ \underline{U_{t_{1}}}(x) = \lim_{j \to \infty} \underline{u}(x,t_{1}+ j\epsilon^{2}/2)$$
for $x \in \Omega $.
By definition of $\underline{u} $, we see that
\begin{align*}
\underline{U_{t_{1}}}(x)    =(1-\overline{\delta}(x))\lim_{j \to \infty} \bigg[ \midrg_{\nu \in S^{n-1}} \mathscr{A}_{\epsilon}\underline{u} \bigg( x,   t_{1}+
 \frac{j \epsilon^{2}}{2}; \nu \bigg) \bigg]  
 + \overline{\delta}(x) \psi(x).
\end{align*}
Therefore, it is sufficient to show that
\begin{align} \label{ac_sup} \lim_{j \to \infty} \sup_{\nu \in S^{n-1}} \mathscr{A}_{\epsilon}\underline{u} \bigg( x,  t_{1}+ \frac{j \epsilon^{2}}{2}; \nu \bigg) = \sup_{\nu \in S^{n-1}} \tilde{\mathscr{A}}_{\epsilon}\underline{U_{t_{1}}} ( x; \nu )
\end{align}
and
\begin{align} \label{ac_inf} \lim_{j \to \infty} \inf_{\nu \in S^{n-1}} \mathscr{A}_{\epsilon}\underline{u} \bigg( x,  t_{1}+ \frac{j \epsilon^{2}}{2}; \nu \bigg) = \inf_{\nu \in S^{n-1}} \tilde{\mathscr{A}}_{\epsilon}\underline{U_{t_{1}}} ( x; \nu )
\end{align}
where 
\begin{align} \label{tildea} \tilde{\mathscr{A}}_{\epsilon}v (x; \nu)= \alpha v(x+ \epsilon \nu) + \beta \kint_{B_{\epsilon}^{\nu} }v(x+ h) d \mathcal{L}^{n-1}(h) . 
\end{align}
These equalities can be derived by the argument in the proof of \cite[Proposition 3.3]{MR3623556}.
First, we get \eqref{ac_sup} from monotonicity of  $\{  \underline{u} ( x,  t_{1}+ j \epsilon^{2}/2  ) \} $.
On the other hand, by means of the monotonicity of $\{  \underline{u} ( x,  t_{1}+ j \epsilon^{2}/2  ) \} $ and continuity of $ \mathscr{A}_{\epsilon}\underline{u} (x, t; \cdot)$, we can show the existence of a vector $\tilde{\nu} \in  S^{n-1}$ satisfying
$$ \mathscr{A}_{\epsilon}\underline{u} \bigg( x,  t_{1}+ \frac{j \epsilon^{2}}{2}; \tilde{\nu} \bigg) \le \lim_{j \to \infty} \inf_{\nu \in S^{n-1}} \tilde{\mathscr{A}}_{\epsilon}\underline{U_{t_{1}}} ( x; \nu ) \qquad \textrm{for any} \ j \ge 0 . $$
Now \eqref{ac_inf} is obtained by the monotone convergence theorem.
Thus, we deduce that
$ \underline{U_{t_{1}}}$ satisfies the DPP \eqref{dpplim} for every $- \epsilon^{2}/2 \le t_{1} < 0  $.
By uniqueness of solutions to \eqref{dpplim}, \cite[Theorem 3.7]{MR3623556}, we can deduce that 
$$ \lim_{t \to \infty} \underline{u}(x,t )=  U_{\epsilon}(x) .  $$

We can prove the same result for $\overline{u}$ by repeating the above steps. 
Combining these results with $\underline{u} \le u_{\epsilon} \le \overline{u}$, we get
$$\lim_{t \to \infty} u_{\epsilon}(x,t )=  U_{\epsilon}(x) $$
and then we can finish the proof.
\end{proof}


We finish this section by proving a corollary.
One can apply the above theorem with the interior regularity result for $ u_{\epsilon} $, \cite[Theorem 5.2]{MR4153524}. 
This coincides with the result for elliptic case, \cite[Theorem 1.1]{MR4125101}.

\begin{corollary} 
Let $\bar{B}_{2r} \subset \Omega \backslash I_{\epsilon}$  and $\epsilon > 0 $ be small. 
Suppose that $U_{\epsilon}$ satisfies \eqref{dpplim}.
Then for any $x, y \in B_{r}(0)$, 
$$ |U_{\epsilon} (x) - U_{\epsilon} (y) | \le C  ( |x-y|  + \epsilon), $$ 
where $C>0$ is a constant which only depends on $r,n $ and $||\psi||_{L^{\infty}(\Gamma_{\epsilon})}$.
\end{corollary}
\begin{proof}
Let $r>0$ with $\bar{B}_{2r} \subset \Omega \backslash I_{\epsilon}$ and  $x, y  \in B_{r}(0)$.
By Theorem \ref{stable}, for any $\eta >0$, we can find some large $t>0$ such that
$$ |u_{\epsilon}(x,t) - U_{\epsilon}(x)| < \eta \quad \textrm{and} \quad |u_{\epsilon}(y,t) - U_{\epsilon}(y)| < \eta ,$$
where $u_{\epsilon}$ is a function satisfying \eqref{dppvar}.
And by  \cite[Theorem 5.2]{MR4153524}, we know that
$$ |u_{\epsilon} (x,t) - u_{\epsilon} (y,t) | \le C( |x-y|+ \epsilon), $$  
where $C$ is a constant depending on $r, n $ and $ ||F||_{L^{\infty}(\Gamma_{\epsilon,T})}$.
(Here, $F$ is a boundary data as in Theorem \ref{stable})

Then we have
\begin{align*}
| U_{\epsilon}(x) -  U_{\epsilon}(y) | & \le | U_{\epsilon}(x)-u_{\epsilon}(x,t)| + 
|u_{\epsilon}(x,t) - u_{\epsilon}(y,t) | + |u_{\epsilon}(y,t) -U_{\epsilon}(y)|  \\
& < C( |x-y|+ \epsilon) + 2\eta.
\end{align*}
Since we can choose $\eta$ arbitrarily small, we obtain
$$  |u_{\epsilon} (x,t) - u_{\epsilon} (y,t) | \le C( |x-y|+ \epsilon) $$
for some $C =  C(n,p, \Omega,  ||\psi||_{L^{\infty}(\Gamma_{\epsilon})})>0$
since we can estimate $$ ||F||_{L^{\infty}(\Gamma_{\epsilon,T})} \le  ||\psi||_{L^{\infty}(\Gamma_{\epsilon})} $$ by choosing proper boundary data $F$.
\end{proof}

\section{Regularity near the boundary}
We consider regularity for functions $u_{\epsilon}$ satisfying \eqref{dppvar} near the boundary in this section. 
This result is interesting in itself,
but it is also necessary to observe the connection between value functions and PDEs
(see the next section).

First, we introduce a boundary regularity condition for the domain $\Omega$.
\begin{definition}[Exterior sphere condition] \label{exspc} We say that a domain $ \Omega$ satisfies
an exterior sphere condition if
for any $y \in \partial \Omega$, there exists $ B_{\delta}(z) \subset \mathbb{R}^{n} \backslash \Omega $ with $\delta > 0 $ such that $y \in  \partial B_{\delta}(z)$.
\end{definition}
Throughout this section, we always assume that $\Omega$ satisfies Definition \ref{exspc} and 
$\Omega \subset B_{R}(z) $ for some $R>0$.

Meanwhile, we also assume that the boundary data $F$ satisfies
\begin{align} \label{bdlip}
|F(x,t)-F(y,s)| \le L(|x-y|+|t-s|^{1/2})
\end{align}
for any $(x,t),(y,s) \in \Gamma_{\epsilon,T}$ and some $L>0$.

Let $y \in \partial \Omega$ and take $z \in \mathbb{R}^{n} \backslash \Omega $ with $ B_{\delta}(z) \subset \mathbb{R}^{n} \backslash \Omega $ and $y \in \partial B_{\delta}(z)$. 
We consider a time-independent tug-of-war game.  
Assume that the rules to move the token are the same as that of the original game,
but of course, we do not consider the time parameter $t$ in this case.
We also assume that the token cannot escape outside $\overline{B}_{R}(z)$
and the game ends only if the token is located in $\overline{B}_{\delta}(z)$. 
Now we fix specific strategies for both players.
For each $k=0, 1, \dots$,
assume that Player I and II takes the vector $\nu_{k}^{\mathrm{I}}=-\frac{x_{k}-z}{|x_{k}-z|} $ and $\nu_{k}^{\mathrm{II}}=\frac{x_{k}-z}{|x_{k}-z|} $, respectively. 
We write these strategies for Player I, II as $S_{\mathrm{I}}^{z}$ and $S_{\mathrm{II}}^{z}$.
On the other hand, we need to define strategies and random processes when $B_{\epsilon}(x_{k}) \backslash B_{R}(z)\neq \varnothing $.   
In this case, $x_{k+1}$ is defined by $x_{k}+\epsilon \nu_{k}^{\mathrm{I}} $
if Player I wins coin toss twice and
 $$x_{k}+ \dist (x_{k}, \partial B_{R}(z))   \nu_{k}^{\mathrm{II}}=z + R\frac{x_{k}-z}{|x_{k}-z|}$$
if Player II wins coin toss twice.
When random walk occurs, $x_{k+1}$ is chosen uniformly in $B_{\epsilon}^{\nu_{k}^{I}}(x) \cap B_{R}(z)$. 
 
We denote by $$ \tau^{\ast} = \inf \{ k : x_{k} \in \overline{B}_{\delta}(z)  \}.$$
The following lemma gives an estimate for the the stopping time $\tau^{\ast}$.  
\begin{lemma} \label{bafn} 
Under the setting as above, we have
$$ \mathbb{E}_{S_{\mathrm{I}}^{z},S_{\mathrm{II}}^{z}}^{x_{0}}[\tau^{\ast}] \le    \frac{C(n, \alpha, R/\delta)( \dist(\partial B_{\delta}(y),x_{0})+o(1))}{\epsilon^{2}}   $$
for any $ x_{0}  \in \Omega  \subset  B_{R}(z)\backslash \overline{B}_{\delta}(z) $.
Here $o(1) \to 0$ as $\epsilon \to 0$ and $\ceil{x}$ means the least integer greater than or equal to $x \in \mathbb{R}$.

\end{lemma} 

\begin{proof}
Set $g_{\epsilon}(x) =  \mathbb{E}_{S_{\mathrm{I}}^{z},S_{\mathrm{II}}^{z}}^{x}[\tau^{\ast}].$
Then we observe that $g_{\epsilon}$ satisfies the following DPP
\begin{align*}
 g_{\epsilon}(x) =& \frac{1}{2}
\bigg[ \bigg\{ \alpha g_{\epsilon} (x+ \rho_{x}\epsilon \nu_{x}) + \beta \kint_{B_{\epsilon}^{\nu_{x}}(x) \cap B_{R}(z)}g_{\epsilon} (y) d \mathcal{L}^{n-1}(y) \bigg\}   \\ &  \qquad
+\bigg\{ \alpha g_{\epsilon} (x- \epsilon \nu_{x})+\beta \kint_{B_{\epsilon}^{\nu_{x}}(x) \cap B_{R}(z)} g_{\epsilon}(y) d \mathcal{L}^{n-1}(y)  \bigg\} \bigg]+1,
\end{align*}
where $\rho_{x} = \min \{ 1, \epsilon^{-1}\dist(x, \partial B_{R}(z)) \}$ and $\nu_{x} = (x-z)/|x-z|$.
Note that $\rho_{x}=1$ for any $x \in  B_{R-\epsilon}(z)\backslash \overline{B}_{\delta}(z) $.
Next we define
$ v_{\epsilon}= \epsilon^{2}g_{\epsilon} $. 
It is straightforward that
\begin{align} \label{vdpp} \begin{split} v_{\epsilon}(x)=\frac{\alpha}{2} \big(v_{\epsilon}&(x+\rho_{x}\epsilon\nu_{x})+v_{\epsilon}(x-\epsilon\nu_{x})\big)
\\ & + \beta \kint_{B_{\epsilon}^{\nu_{x}}(x) \cap B_{R}(z)} v_{\epsilon}(y) d \mathcal{L}^{n-1}(y)  +\epsilon^{2}. 
\end{split}
\end{align}  
From the definition of $v_{\epsilon}$ and \eqref{vdpp}, we observe that the function $v_{\epsilon}$ is rotationally symmetric, that is, $v_{\epsilon}$ is a function of $r = |x-z|$.
If we denote by $v_{\epsilon}(x)=V(r)$, the DPP \eqref{vdpp} can be represented by 
\begin{align} \label{vdpp2} \begin{split}
V(r) = \frac{\alpha}{2} &\big(V(r+\rho_{r}\epsilon)+V(r-\epsilon)\big)
\\ & + \beta \kint_{B_{\epsilon}^{\nu_{x}}(x) \cap B_{R}(z)} V(|y-z|) d \mathcal{L}^{n-1}(y)  +\epsilon^{2},
\end{split}
\end{align}
where $\rho_{r}=\min \{ 1, \epsilon^{-1}(R-r) \}$.

Now we can deduce that \eqref{vdpp2} has a connection to the following problem
\begin{align*} 
\left\{ \begin{array}{ll}
\frac{1-\alpha}{2r} \frac{n-1}{n+1}w'+\frac{\alpha}{2}w''= -1 & \textrm{when $ r \in (\delta, R+\epsilon), $}\\
w(\delta)=0, \\
w'(R+\epsilon)=0\\
\end{array} \right. 
\end{align*}
by using Taylor expansion.
Note that if we set $v(x)=w(|x|)$,
$$\frac{1-\alpha}{2r} \frac{n-1}{n+1}w'+\frac{\alpha}{2}w''= -1$$ can be transformed by
$$ \Delta_{p}^{N}v = -2(p+n),$$
where $p=(1+n\alpha)/(1-\alpha)$ (for the definition of $ \Delta_{p}^{N}$, see the next section).
On the other hand, we have 
\begin{align*}w(r)= 
\left\{ \begin{array}{ll}
-\frac{n+1}{2\alpha +n-1}r^{2}+c_{1}r^{\frac{2\alpha n-n+1}{(n+1)\alpha}}+c_{2} & \textrm{when $ \alpha \neq \frac{n-1}{2n}, $}\\
-\frac{n}{n-1}r^{2}+c_{1}\log r+c_{2} & \textrm{when $ \alpha = \frac{n-1}{2n}  $}\\
\end{array} \right. 
\end{align*}
by direct calculation. 
Here 
\begin{align*}c_{1}= 
\left\{ \begin{array}{ll}
\frac{2(n+1)^{2}\alpha}{(2\alpha +n-1)(2\alpha n -n+1)}(R+\epsilon)^{\frac{n+2\alpha-1}{(n+1)\alpha}} & \textrm{when $ \alpha \neq \frac{n-1}{2n}, $}\\
\frac{2n}{n-1}(R+\epsilon)^{2} & \textrm{when $\alpha = \frac{n-1}{2n} $}\\
\end{array} \right. 
\end{align*} is positive if $\alpha \ge  \frac{n-1}{2n}$ and negative otherwise.
We extend this function to the interval $(\delta -\epsilon, R+\epsilon]$. 

Observe that 
\begin{align*} 
\frac{\alpha}{2} & \big(w(r+\epsilon)+w(r-\epsilon)\big)  + \beta\kint_{B_{\epsilon}^{\nu_{x}}(x) } w(|y-z|) d \mathcal{L}^{n-1}(y) \\
& = w(r) -  \frac{n+1}{2\alpha+n-1} \bigg( \alpha + \frac{n-1}{n+1}\beta \bigg)   \epsilon^{2}  +o(\epsilon^{2}) \\ 
& \le  w(r) - \bigg[  \frac{n+1}{2\alpha+n-1} \bigg( \alpha + \frac{n-1}{n+1}\beta \bigg)- \eta  \bigg] \epsilon^{2}
\end{align*}
for some $\eta>0$ when $\alpha \neq  \frac{n-1}{2n}$ (we can also obtain a similar estimate if $\alpha= \frac{n-1}{2n}$).  
Set $$c_{0}:=\frac{n+1}{2\alpha+n-1} \bigg( \alpha + \frac{n-1}{n+1}\beta \bigg)- \eta  >0.$$
Then we have 
\begin{align*}
\mathbb{E}_{S_{\mathrm{I}}^{z},S_{\mathrm{II}}^{z}}^{x_{0}} &[v(x_{k})+ kc_{0}\epsilon^{2}|x_{0}, \dots, x_{k-1}] \\ &
 = \alpha   \big(v(x_{k-1}+\epsilon \nu_{x_{k-1}})+v(x_{k-1}-\epsilon\nu_{x_{k-1}})\big) \\ 
& \qquad \qquad  + \beta\kint_{B_{\epsilon}^{\nu_{x_{k-1}}}(x_{k-1}) } v(y-z ) d \mathcal{L}^{n-1}(y)  +kc_{0}\epsilon^{2}
\\ & \le  v(x_{k-1})+ (k-1)c_{0}\epsilon^{2},
\end{align*}
if $B_{\epsilon}(x_{k-1}) \subset B_{R}(z) \backslash \overline{B}_{\delta-\epsilon}(z) $.
The same estimate can be derived in the case $B_{\epsilon}(x_{k-1}) \backslash B_{R}(z)\neq \varnothing $ since $ w$ is an increasing function of $r$ and it implies
$$ v(x+\rho_{x}\epsilon\nu_{x}) \le v(x+\epsilon\nu_{x}) $$
and
$$ \kint_{B_{\epsilon}^{\nu_{x}}(x) \cap B_{R}(y) } v( y-z ) d \mathcal{L}^{n-1}(y) \le \kint_{B_{\epsilon}^{\nu_{x}}(x) } v( y-z ) d \mathcal{L}^{n-1}(y) .$$  

Now we see that $v(x_{k})+ kc_{0}\epsilon^{2} $ is a supermartingale.
By the optional stopping theorem, we have
\begin{align} \label{ostapp} \mathbb{E}_{S_{\mathrm{I}}^{z},S_{\mathrm{II}}^{z}}^{x_{0}}[v(x_{\tau^{\ast} \wedge k })+(\tau^{\ast} \wedge k)c_{0}\epsilon^{2}] \le v(x_{0}).\end{align}
We also check that 
$$ 0 \le -\mathbb{E}_{S_{\mathrm{I}}^{z},S_{\mathrm{II}}^{z}}^{x_{0}}[v(x_{\tau^{\ast}})] \le o(1) ,$$
since $x_{\tau^{\ast}} \in \overline{B}_{\delta}(z) \backslash \overline{B}_{\delta-\epsilon}(z)$.

Meanwhile, it can be also observed that $w'>0$ is a decreasing function in the interval $(\delta, R+\epsilon)$
and thus 
$$ w' \le \frac{2(n+1)}{2\alpha+n-1}\delta \bigg[ \bigg( \frac{R+\epsilon}{\delta} \bigg)^{\frac{n+2\alpha-1}{(n+1)\alpha}} -1 \bigg] 
\qquad \textrm{in} \  (\delta, R+\epsilon).$$
From the above estimate, we have
\begin{align} \label{linestw} 0 \le w(x_{0}) \le C(n, \alpha, R/\delta) \dist( \partial B_{\delta}(y), x_{0} ). 
\end{align}
Finally, combining \eqref{linestw} with \eqref{ostapp}  and passing to a limit with $k$, we have
\begin{align*}
c_{0}\epsilon^{2}\mathbb{E}_{S_{\mathrm{I}}^{z},S_{\mathrm{II}}^{z}}^{x_{0}}[ \tau^{\ast} ] & \le w(x_{0}) -\mathbb{E}_{S_{\mathrm{I}}^{z},S_{\mathrm{II}}^{z}}^{x_{0}}[w(x_{\tau^{\ast}})] \\ &
\le C(n, \alpha, R/\delta)  \dist( \partial B_{\delta}(y), x_{0} )+ o(1)
\end{align*}
and it gives our desired estimate.
\end{proof}
\color{black}

By means of Lemma \ref{bafn}, we can deduce following boundary regularity results. 
First, we give an estimate for $u_{\epsilon}$ on the lateral boundary.
\begin{lemma} \label{latbest}
Assume that $\Omega$ satisfies the exterior sphere condition and $F$ satisfies \eqref{bdlip}.
Then for the value function $u_{\epsilon}$ with boundary data $F$, we have
\begin{align}\label{towlatest} \begin{split}
& |u_{\epsilon}(x,t)-u_{\epsilon}(y,s)| 
\\ & \le C(n,\alpha, R,\delta, L) (K+ K^{1/2}) +L(|x-y|+|t-s|^{1/2}+ 2\delta),
\end{split}
\end{align} 
where $K =\min \{ |x-y|,t \}+\epsilon$ and $R, \delta$ are the constants in Lemma \ref{bafn}
for every $(x,t) \in \Omega_{T}$ and $(y,s) \in O_{\epsilon,T}$.
\end{lemma}
\begin{proof}
We first consider the case $t=s$.
Set $N = \ceil{ 2t/\epsilon^{2}} $. 
Since $\Omega$ satisfies the exterior sphere condition,
we can find a ball $B_{\delta}(z) \subset \mathbb{R}^{n} \backslash \Omega $
such that $y \in \partial B_{\delta}(z)$.
Assume that Player I takes a strategy $S_{I}^{z}$ of pulling towards $z$.

We estimate the expected value for the distance $|x_{\tau}-x_{0}|$ under the game setting.
Let $\theta$ be the angle between $\nu$ and $x-z$. 
And we assume that $x=0$ and $z=(0,\cdots,0, r \sin\theta, -r\cos\theta)$
by using a proper transformation. 
Then the following term
$$ \alpha |x+\epsilon \nu-z|  + \beta \kint_{B_{\epsilon}^{\nu_{x}}(x ) 
 }|\tilde{x}-z| d\mathcal{L}^{n-1}(\tilde{x})$$
can be written as
\begin{align*} &A(\theta)\\ & = \alpha \sqrt{(r \sin \theta)^{2}+(r \cos \theta+\epsilon)^{2} }
+ \beta \kint_{T_{\epsilon}} \sqrt{(y-r\sin \theta)^{2}+(r\cos \theta)^{2}} d\mathcal{L}^{n-1}(y)
\\ & = \alpha \sqrt{r^{2}+2r\epsilon \cos \theta + \epsilon^{2}}
+ \beta \kint_{T_{\epsilon}} \sqrt{r^{2}-2ry_{n-1} \sin \theta + |y|^{2}} d\mathcal{L}^{n-1}(y)
\\ & =: \alpha A_{1}(\theta) + \beta A_{2} (\theta) ,
\end{align*}
where $r = |x-z|$ and $T_{\epsilon} = \{ x = (x_{1}, \dots, x_{n}) \in B_{\epsilon}(0) : x_{n}=0 \}$.
Observe that $A_{1}$ is decreasing in the interval $(0, \pi)$. (Thus, $A_{1}$ has the maximum at $\theta=0$ in $[0, \pi]$)
On the other hand, we have
$$ A_{2}^{'}(\theta) = -  \kint_{T_{\epsilon}} \frac{ry_{n-1}\cos \theta}{\sqrt{r^{2}-2ry_{n-1} \sin \theta + |y|^{2}} } d\mathcal{L}^{n-1}(y)$$
and this function is a symmetric function about $\theta= \pi /2 $. 
We also check that $A_{2}^{'} <0$ in $(0, \pi/2)$.
Thus, we verify that 
$A_{2}$ has a maximum at $\theta=0, \pi$ in $[0, \pi]$ and $\theta(0)=\theta(\pi)$. 
This leads to the following estimate
\begin{align} \begin{split}\label{nuest}
\sup_{\nu \in S^{n-1}}\bigg[\alpha |x&+\epsilon \nu -z| + \beta \kint_{B_{\epsilon}^{\nu}(x ) 
 }|\tilde{x}-z| d\mathcal{L}^{n-1}(\tilde{x}) \bigg]
\\ & = \alpha( |x-z|+\epsilon ) + \beta \kint_{B_{\epsilon}^{\nu_{x}}(x ) 
 }|\tilde{x}-z| d\mathcal{L}^{n-1}(\tilde{x}),
 \end{split}
\end{align}
where $\nu_{x} = (x-z)/|x-z|$.

Therefore, we have
\begin{align*}
&\mathbb{E}_{S_{\mathrm{I}}^{z}, S_{\mathrm{II}}}^{(x_{0},t)} [|x_{k} -z| |(x_{0},t_{0}), \dots, (x_{k-1},t_{k-1})] \\
& \le \frac{1-\delta(x_{k-1},t_{k-1})}{2} \bigg[ \alpha(|x_{k-1}-z|-\epsilon) + \beta \kint_{B_{\epsilon}^{\nu_{x_{k-1}}}(x_{k-1} ) } |\tilde{x}-z| d\mathcal{L}^{n-1}(\tilde{x}) \bigg] \\
& \quad + \frac{1-\delta(x_{k-1},t_{k-1})}{2} \bigg[ \alpha(|x_{k-1}-z|+\epsilon) + \beta \kint_{B_{\epsilon}^{\nu_{x_{k-1}}}(x_{k-1} ) }\hspace{-0.6em} |\tilde{x}-z| d\mathcal{L}^{n-1}(\tilde{x}) \bigg] 
\\ & \quad + \delta(x_{k-1},t_{k-1})|x_{k-1}-z|
\\ & =  |x_{k-1}-z| \\ & \qquad + \beta(1-\delta(x_{k-1},t_{k-1})) \bigg( \kint_{B_{\epsilon}^{\nu_{x_{k-1}}}(x_{k-1} ) } \hspace{-0.6em}|\tilde{x}-z| d\mathcal{L}^{n-1}(\tilde{x}) - |x_{k-1}-z| \bigg).
\end{align*}
We also observe that $$0<\beta(1-\delta(x_{k-1},t_{k-1}))<1,$$
$$ |x_{k-1}-z| \le |\tilde{x}-z| \le  \sqrt{(x_{k-1}-z)^{2} + \epsilon^{2}} \qquad \textrm{for}\ x \in B_{\epsilon}^{\nu_{x_{k-1}}},$$
and 
$$ 0< \sqrt{a^{2}+\epsilon^{2}} - a < \frac{\epsilon^{2}}{2a} \qquad \textrm{for}\ a>0. $$

Therefore,
\begin{align*}
\mathbb{E}_{S_{\mathrm{I}}^{z}, S_{\mathrm{II}}}^{(x_{0},t)} [|x_{k} -z| |(x_{0},t_{0}), \dots, (x_{k-1},t_{k-1})] \le |x_{k-1}-z|+ C\epsilon^{2}
\end{align*}
for some $C=C(n,\delta)>0$.
This yields that 
$$ M_{k}= |x_{k}-z| -Ck \epsilon^{2} $$
is a supermartingale.

Applying the optional stopping theorem and Jensen's inequality to $M_{k}$, we derive that
\begin{align}  \label{oriste} \begin{split}
\mathbb{E}_{S_{\mathrm{I}}^{z}, S_{\mathrm{II}}}^{(x_{0},t)}& [|x_{\tau} -z|+|t_{\tau}-t|^{1/2}  ]  \\
& = \mathbb{E}_{S_{\mathrm{I}}^{z}, S_{\mathrm{II}}}^{(x_{0},t)} \bigg[|x_{\tau} -z|+ \epsilon \sqrt{\frac{\tau}{2}} \bigg]
\\ & \le |x_{0} -z| + C \epsilon^{2} \mathbb{E}_{S_{\mathrm{I}}^{z}, S_{\mathrm{II}}}^{(x_{0},t)}[\tau] +C\epsilon \big(\mathbb{E}_{S_{\mathrm{I}}^{z}, S_{\mathrm{II}}}^{(x_{0},t)}[\tau] \big)^{1/2}.
\end{split}
\end{align} 
Next we need to obtain estimates for $\mathbb{E}_{S_{\mathrm{I}}^{z}, S_{\mathrm{II}}}^{(x_{0},t)}[\tau] $. To do this, we use the result in Lemma \ref{bafn}.
We can check that the exit time $\tau$ of the original game is bounded by $\tau^{\ast}$ 
because the expected value of $|x_{k}-z|$ for given $|x_{k-1}-z|$ is maximized when Player II chooses the strategy $S_{\mathrm{II}}^{z}$ from \eqref{nuest}.
Thus, we have
\begin{align*}
\mathbb{E}_{S_{\mathrm{I}}^{z}, S_{\mathrm{II}}}^{(x_{0},t)}[\tau]& \le \min \{
 \mathbb{E}_{S_{\mathrm{I}}^{z},S_{\mathrm{II}}^{z}}^{x_{0}}[\tau^{\ast}] , N \}
 \\ & \le
  \min \bigg\{\frac{C(n,\alpha, R/\delta)  (\dist(\partial B_{\delta}(z),x_{0})+\epsilon)}{\epsilon^{2}} , N \bigg\}
\end{align*}
for any strategy $S_{\mathrm{II}}$ for Player II. 
We also see that
$$ \dist(x_{0}, \partial B_{\delta}(z)) \le |x_{0}-y|. $$
This and \eqref{oriste} imply
\begin{align*}
&\mathbb{E}_{S_{\mathrm{I}}^{z}, S_{\mathrm{II}}}^{(x_{0},t)} [|x_{\tau} -z|+|t_{\tau}-t|^{1/2}  ] \\ & \le
|x_{0}-y| +C \min \{   |x_{0}-y|+\epsilon,  \epsilon^{2}N \}+ C \min \{   |x_{0}-y|+\epsilon,  \epsilon^{2}N \}^{1/2},
\end{align*}
where $C$ is a constant depending on $n,\alpha, R$ and $\delta$.
Therefore, we get
\begin{align*}
|\mathbb{E}_{S_{\mathrm{I}}^{z}, S_{\mathrm{II}}}^{(x_{0},t)} &[F(x_{\tau},t_{\tau})] - F(z,t)| \\ & 
\le L( |x_{0}-y| +C(n,\alpha, R,\delta) \min \{   |x_{0}-y|+\epsilon,  \epsilon^{2}N \} 
\\ & \qquad \quad +  C(n,\alpha, R/\delta)\min \{   |x_{0}-y|+\epsilon,  \epsilon^{2}N \}^{1/2} )
\end{align*}
and this yields
\begin{align*}
u_{\epsilon}(x_{0},t)&=
\sup_{S_{\mathrm{I}}}\inf_{S_{\mathrm{II}}}\mathbb{E}_{S_{\mathrm{I}} , S_{\mathrm{II}}}^{(x_{0},t)} [F(x_{\tau},t_{\tau})] \\ &
\ge \inf_{S_{\mathrm{II}}}\mathbb{E}_{S_{\mathrm{I}}^{z}, S_{\mathrm{II}}}^{(x_{0},t)} [F(x_{\tau},t_{\tau})] \\ &
\ge F(z,t)- L\{ C(n,\alpha, R,\delta) (K+ K^{1/2})+|x_{0}-y|\}
\\ & \ge F(y,t)-C(n,\alpha, R,\delta,L) (K+ K^{1/2}) -  L(|x_{0}-y|+2 \delta)
\end{align*}
for $K = \min \{ |x_{0}-y|+\epsilon,  \epsilon^{2}N \}$.
Note that we can also derive the upper bound for $u_{\epsilon}(x_{0},t)$
by taking the strategy where Player II pulls toward to $z$.

Meanwhile, in the case of $t\neq s$, we have
\begin{align*}
 &|u_{\epsilon}(x,t)-u_{\epsilon}(y,s)| \\ & \le 
|u_{\epsilon}(x,t)-u_{\epsilon}(y,t)|+|u_{\epsilon}(y,t)-u_{\epsilon}(y,s)| \\ & 
\le C(n,\alpha,R/\delta,L)(K+ K^{1/2})+ L (|x-y|+2\delta )+ L|t-s|^{1/2} ,
\end{align*}
where $K = \min \{ |x_{0}-y|+\epsilon,  \epsilon^{2}N \}$ and $N= \ceil {2t/\epsilon^{2}} $.
This gives our desired estimate.
\end{proof}

We can also derive the following result on the initial boundary.
\begin{lemma} \label{inbest}
Assume that $\Omega$ satisfies the exterior sphere condition and $F$ satisfies \eqref{bdlip}.
Then for the value function $u_{\epsilon}$ with boundary data $F$, we have
\begin{align} \label{inbinq}
|u_{\epsilon}(x,t)-u_{\epsilon}(y,s)|  \le C (|x-y|+t ^{1/2}+ \epsilon) 
\end{align} 
for every $(x,t) \in \Omega_{T}$ and $(y,s) \in \Omega \times (-\epsilon^{2}/2,0]$.
The constant $C$ depends only on $ n,L$.
\end{lemma}
\begin{proof}
Set $(x,t)=(x_{0},t_{0})$ and $N =\ceil{ 2t/\epsilon^{2}} $.
As in the above lemma, we also estimate the expected value of the distance between $y$ and the exit point $x_{\tau}$. 
Consider the case that Player I chooses a strategy of pulling to $y$. 
When $|x_{k-1}-y| \ge \epsilon$, we have
\begin{align*}
&\mathbb{E}_{S_{\mathrm{I}}^{y},S_{\mathrm{II}}}^{(x_{0},t_{0})}[|x_{k}-y|^{2}| (x_{0},t_{0}), \dots, (x_{k-1},t_{k-1}) ] \\
& \le (1-\delta (x_{k-1},t_{k-1}))\times \\ & \quad
 \bigg[ \frac{\alpha}{2} \{ (|x_{k-1}-y|+\epsilon)^{2} +(|x_{k-1}-y|-\epsilon)^{2}   \} \\ &
 \qquad \qquad 
+ \beta \kint_{B_{\epsilon}^{\nu_{x_{k-1}}}(x_{k-1} ) }  |\tilde{x}-y|^{2} d\mathcal{L}^{n-1}(\tilde{x}) \bigg]
+ \delta (x_{k-1},t_{k-1})|x_{k-1}-y|^{2}
\\ & \le \alpha (|x_{k-1}-y|^{2}+\epsilon^{2}) + \beta (|x_{k-1}-y|^{2}+C\epsilon^{2})
\\ & \le |x_{k-1}-y|^{2}+C\epsilon^{2}
\end{align*}
for some constant $C>0$ which is independent of $\epsilon$. 
We recall the notation $\nu_{x_{k-1}} = (x_{k-1}-z)/|x_{k-1}-z| $ here.
Otherwise, we also see that
\begin{align*}
&\mathbb{E}_{S_{\mathrm{I}}^{y},S_{\mathrm{II}}}^{(x_{0},t_{0})}[|x_{k}-y|^{2}| (x_{0},t_{0}), \dots, (x_{k-1},t_{k-1}) ] \\
& \le (1-\delta (x_{k-1},t_{k-1}))
 \bigg[ \frac{\alpha}{2} (|x_{k-1}-y|+\epsilon)^{2} 
+ \beta \kint_{B_{\epsilon}^{\nu_{x_{k-1}}}(x_{k-1} ) }\hspace{-0.7em}  |\tilde{x}-y|^{2} d\mathcal{L}^{n-1}(\tilde{x}) \bigg]
\\ & \quad + \delta (x_{k-1},t_{k-1})|x_{k-1}-y|^{2},
\end{align*}
and then we get the same estimate as above since  
$$(|x_{k-1}-y|+\epsilon)^{2}  \le 2(|x_{k-1}-y|^{2} +\epsilon^{2}). $$
Therefore, we see that $$M_{k} = |x_{k}-y|^{2}-Ck\epsilon^{2}$$
is a supermartingale.

Now we obtain 
$$\mathbb{E}_{S_{\mathrm{I}}^{y},S_{\mathrm{II}}}^{(x_{0},t)}[|x_{\tau}-y|^{2}] \le 
|x_{0}-y|^{2} + C \epsilon^{2} \mathbb{E}_{S_{\mathrm{I}}^{y},S_{\mathrm{II}}}^{(x_{0},t)}[\tau] $$
by using the optional stopping theorem.
Since $\tau < \ceil{ 2t/\epsilon^{2}}$, the right-hand side term is estimated by
$ |x_{0}-y|^{2} +C (t+ \epsilon^{2}) $.
Applying Jensen's inequality, we get
\begin{align*}
\mathbb{E}_{S_{\mathrm{I}}^{y},S_{\mathrm{II}}}^{(x_{0},t)}[|x_{\tau}&-y| ] \le 
\big( \mathbb{E}_{S_{\mathrm{I}}^{y},S_{\mathrm{II}}}^{(x_{0},t)}[|x_{\tau}-y|^{2}] \big)^{\frac{1}{2}} 
\\ & \le \big( |x_{0}-y|^{2} + C (t + \epsilon^{2}) \big)^{\frac{1}{2}} 
\\ & \le |x_{0}-y| + C(t^{1/2}+ \epsilon ).
\end{align*}
From the above estimate, we deduce that
\begin{align*}
u_{\epsilon}(x_{0},t)&=
\sup_{S_{\mathrm{I}}}\inf_{S_{\mathrm{II}}}\mathbb{E}_{S_{\mathrm{I}} , S_{\mathrm{II}}}^{(x_{0},t)} [F(x_{\tau},t_{\tau})] \\ &
\ge F(y, t) - L \ \mathbb{E}_{S_{\mathrm{I}}^{y},S_{\mathrm{II}}}^{(x_{0},t)}[|x_{\tau}-y|+|t-t_{\tau}|^{1/2}]
\\ & \ge F(y,t) - C(|x_{0}-y|+t^{1/2}+\epsilon).
\end{align*}

The upper bound can be derived in a similar way, and then we get the estimate \eqref{inbinq}.
\end{proof}

\color{black}
  
\section{Application to PDEs}
The objective of this section is to study behavior of $u_{\epsilon}$ when $\epsilon$ tends to zero. 
This issue has been studied in several preceding papers (see \cite{MR2875296,MR3161602,MR3494400,MR3623556}).
Those results show that there is a close relation between value functions of tug-of-war games and certain types of PDEs.
Now we will establish that there is a convergence theorem showing that $u_{\epsilon}$ converge to the unique viscosity solution of the following Dirichlet problem for the normalized parabolic $p$-Laplace equation
\begin{align} \label{paraplap}
\left\{ \begin{array}{ll}
(n+p)u_{t}= \Delta_{p}^{N} u & \textrm{in $\Omega_{T} $,}\\
 u=F & \textrm{on $\partial_{p}\Omega_{T}$}\\
\end{array} \right. 
\end{align}
as $ \epsilon \to 0$.
Here, $p$ satisfies $\alpha = (p-1)/(p+n) $ and $ \beta = (n+1)/(p+n)$.

Now we introduce the notion of viscosity solutions for \eqref{paraplap}.
Note that we need to consider the case when the gradient vanishes. 
Here we use semicontinuous extensions of operators in order to define viscosity solutions.
For these extensions, we refer the reader to \cite{MR1770903,MR2238463} for more details.
\begin{definition}[Viscosity solution] \label{vissol} 
A function $u \in C(\Omega_{T}) $ is a viscosity solution to \eqref{paraplap} if
the following conditions hold:
\begin{itemize}
\item[(a)] for all $ \varphi \in C^{2}(\Omega_{T} ) $ touching $u$ from above at $(x_{0},t_{0}) \in \Omega_{T} $, 
\begin{align*} 
\left\{ \begin{array}{ll}
 \Delta_{p}^{N}\varphi(x_{0},t_{0})  \ge (n+p) \varphi_{t}(x_{0},t_{0}) \qquad \qquad \textrm{if $ D\varphi (x_{0},t_{0}) \neq 0 $,}\\
\lambda_{\max}((p-2)D^{2}\varphi(x_{0},t_{0}) ) & \\ \qquad \qquad \qquad + \Delta\varphi(x_{0},t_{0}) \ge (n+p) \varphi_{t}(x_{0},t_{0})   \qquad \textrm{if $D\varphi (x_{0},t_{0}) = 0   $.}\\
\end{array} \right. 
\end{align*}
\item[(b)] for all $ \varphi \in C^{2}(\Omega_{T} ) $ touching $u$ from below at $(x_{0},t_{0}) \in \Omega_{T}$,
\begin{align*} 
\left\{ \begin{array}{ll}
 \Delta_{p}^{N}\varphi(x_{0},t_{0})  \le (n+p) \varphi_{t}(x_{0},t_{0}) \qquad \qquad \textrm{if $ D\varphi (x_{0},t_{0}) \neq 0 $,}\\
\lambda_{\min}((p-2)D^{2}\varphi(x_{0},t_{0}) ) & \\ \qquad \qquad \qquad + \Delta\varphi(x_{0},t_{0}) \le (n+p) \varphi_{t}(x_{0},t_{0})   \qquad \textrm{if $D\varphi (x_{0},t_{0}) = 0   $.}\\
\end{array} \right. 
\end{align*}
\end{itemize}
Here, the notation $\lambda_{\max}(X)$ and $\lambda_{\min}(X)$ mean the largest and the smallest eigenvalues of a symmetric matrix $X$. 
\end{definition}

The following Arzel\`a-Ascoli criterion will be used to obtain the main result in this section.
It is essentially the same proposition as \cite[Lemma 5.1]{MR3494400}.
We can find the proof of this criterion for elliptic version in \cite[Lemma 4.2]{MR3011990}.
\begin{lemma} \label{aras}
Let $ \{ u_{\epsilon} : \overline{\Omega}_{T} \to \mathbb{R}, \epsilon >0 \} $ be a set of functions such that
\begin{itemize}
\item[(a)] there exists a constant $C >0$ so that $|u_{\epsilon}(x,t)| < C $ for every $ \epsilon > 0$ and every $(x,t) \in \overline{\Omega}_{T} $. 
\\ \item[(b)] given $ \eta > 0$, there are constants $r_{0}$ and $\epsilon_{0}$ so that for every $\epsilon >0$ and $(x,t),(y,s) \in \overline{\Omega}_{T}$ with $d ((x,t),(y,s)) < r_{0} $, it holds
$$ |u_{\epsilon}(x,t)-u_{\epsilon}(y,s)| < \eta .$$
\end{itemize}
Then, there exists a uniformly continuous function $u: \overline{\Omega}_{T} \to \mathbb{R} $ and a subsequence $\{ u_{\epsilon_{i}} \} $ such that $ u_{\epsilon_{i}}  $ uniformly converge to $u$  in $\overline{\Omega}_{T}$, as $i \to \infty$.
\end{lemma}

Now we can describe the relation between functions satisfying \eqref{dppvar} and solutions to the normalized parabolic $p$-Laplace equation.
\begin{theorem}
Assume that $\Omega$ satisfies the exterior sphere condition and  $ F \in C(\Gamma_{\epsilon,T}) $
satisfies \eqref{bdlip}.
Let $u_{\epsilon}$ denote the solution to \eqref{dppvar} with boundary data $ F  $ for each $\epsilon>0$.
Then, there exist a function $ u: \overline{\Omega}_{\epsilon,T} \to \mathbb{R}$ and a subsequence 
$ \{ \epsilon_{i} \} $ such that 
$$ u_{\epsilon_{i}} \to u \qquad \textrm{uniformly in} \quad \overline{\Omega}_{T}$$
and the function $u$ is a unique viscosity solution to \eqref{paraplap}.
\end{theorem}
\begin{remark} 
The uniqueness of solutions to \eqref{paraplap} can be found in \cite[Lemma 6.2]{MR3494400}.
\end{remark}
\begin{proof}
First we check that there is a subsequence $ \{ u_{\epsilon_{i}} \} $  with $ u_{\epsilon_{i}}$ converge uniformly to $u$ on $\overline{\Omega}_{T}$ for some function $u$.
By using the definition of $u_{\epsilon} $, we have
 $$||u_{\epsilon} ||_{L^{\infty}(\Omega_{T})} \le ||F||_{L^{\infty}(\Omega_{T})} < \infty $$
for any $\epsilon > 0$. Hence, $u_{\epsilon} $ are uniformly bounded.  
By means of Lemma \ref{latbest}, Lemma \ref{inbest} and the interior regularity result \cite[Theorem 5.2]{MR4153524}, we know that $ \{ u_{\epsilon} \} $
are equicontinuous.  
Therefore, we can find a subsequence $\{ u_{\epsilon_{i}} \}_{i=1}^{\infty} $ converging uniformly to a function $u \in C(\overline{\Omega}_{T})$
by Lemma \ref{aras}.

Now we need to show that $u$ is a viscosity solution to \eqref{paraplap}.
On the parabolic boundary, we see that
$$ u(x,t) = \lim_{i \to \infty}u_{\epsilon_{i}}(x,t) = F(x,t)$$
for any $(x,t) \in \partial_{p}\Omega_{T}$.

Next we prove that $u$ satisfies 
$$(n+p)u_{t}= \Delta_{p}^{N} u  \qquad \textrm{in} \ \Omega_{T}$$
in the viscosity sense.
Without loss of generality, it suffices to show that $u$ satisfies condition (a) in Definition \ref{vissol}. 

Fix $(x,t) \in \Omega_{T}$.
Then there is a small number $R>0$ such that 
$$ Q:= (x_{0},t_{0}) + B_{R}(0) \times (-R^{2},0) \subset \subset \Omega_{T} .$$
We also assume that $\epsilon>0$ satisfies $Q \subset \Omega_{T} \backslash I_{\epsilon,T}$.
Suppose that a function $\varphi \in C^{2}(Q)$ touches $u$ from below at $(x,t)$.
Then we observe that
$$ \inf_{Q } (u- \varphi) = (u-\varphi)(x,t) \le (u-\varphi)(z,s)  $$
for any $(z,s)  \in Q $.
Since $u_{\epsilon}$ converge uniformly to $u$, for sufficiently small $\epsilon >0$,
there is a point $ (x_{\epsilon},t_{\epsilon}) \in Q$  such that 
$$ \inf_{Q } (u_{\epsilon}- \varphi) \le (u_{\epsilon}-\varphi)(z,s)  $$
for any $(z,s)   \in Q  $.
We also check that $(x_{\epsilon}, t_{\epsilon}) \to (x,t)$ as $ \epsilon \to 0$.

Recall \eqref{deft}. Since $(x_{\epsilon}, t_{\epsilon}) \in \Omega_{T} \backslash I_{\epsilon,T}$,
we have
\begin{align*}
Tu(x,t) =\midrg_{\nu \in S^{n-1}} \mathscr{A}_{\epsilon}u \bigg( x, t- \frac{\epsilon^{2}}{2} ; \nu \bigg).
\end{align*}
We also set $ \psi = \varphi +( u_{\epsilon}-\varphi)(x_{\epsilon}, t_{\epsilon})$
and observe that $u_{\epsilon} \ge \psi$ in $Q$.
Now it can be checked that
\begin{align*}
u_{\epsilon}(x_{\epsilon}, t_{\epsilon})=Tu_{\epsilon}(x_{\epsilon}, t_{\epsilon}) 
 \ge T\psi(x_{\epsilon}, t_{\epsilon}) 
\end{align*}
and
\begin{align*}
 T\psi(x_{\epsilon}, t_{\epsilon}) =T\varphi(x_{\epsilon}, t_{\epsilon})+( u_{\epsilon}-\varphi)(x_{\epsilon}, t_{\epsilon})
\end{align*}
Therefore,
\begin{align*}
u_{\epsilon}(x_{\epsilon}, t_{\epsilon}) \ge T\varphi(x_{\epsilon}, t_{\epsilon})+( u_{\epsilon}-\varphi)(x_{\epsilon}, t_{\epsilon})
\end{align*}
and this implies
\begin{align} \label{2est1} 0  \ge T\varphi (x_{\epsilon}, t_{\epsilon}) - \varphi (x_{\epsilon}, t_{\epsilon}). 
\end{align}

On the other hand, by the Taylor expansion, we observe that
\begin{align*} \frac{1}{2} & \bigg[ \varphi \bigg(x+\epsilon \nu, t-\frac{\epsilon^{2}}{2} \bigg)+ \varphi\bigg(x-\epsilon \nu, t-\frac{\epsilon^{2}}{2} \bigg) \bigg] \\ &
= \varphi (x,t) -\frac{\epsilon^{2}}{2} \varphi_{t}(x,t)+ \frac{\epsilon^{2}}{2} \langle D^{2} \varphi (x,t) \nu, \nu \rangle + o( \epsilon^{2})
\end{align*}
and 
\begin{align*}
\kint_{B_{\epsilon}^{\nu} } &\varphi \bigg(x+ h,t-\frac{\epsilon^{2}}{2} \bigg) d \mathcal{L}^{n-1}(h) \\ & =  \varphi (x,t) -\frac{\epsilon^{2}}{2} \varphi_{t}(x,t)+ \frac{\epsilon^{2}}{2(n+1)}
\Delta_{\nu^{\perp}} \varphi (x,t) +  o( \epsilon^{2})
\end{align*}
where $$ \Delta_{\nu^{\perp}} \varphi (x,t) = \sum_{i=1}^{n-1} \langle D^{2} \varphi (x,t) \nu_{i}, \nu_{i} \rangle$$ with 
$ \nu_{1}, \cdots, \nu_{n-1} $ the orthonormal basis for the space $\nu^{\perp} $
for $ \nu \in S^{n-1} $.

We already know that $  \mathscr{A}_{\epsilon} \varphi $ is continuous with respect to $\nu$ in Proposition \ref{contilem}.
Therefore, there exists a vector $\nu_{\min} = \nu_{\min}(\epsilon) $ minimizing 
$  \mathscr{A}_{\epsilon} \varphi (x_{\epsilon, \eta}, t_{\epsilon}; \cdot ) .$
Then we can calculate
\begin{align*}
& T\varphi(x_{\epsilon}, t_{\epsilon})
\\ & \ge  \frac{\alpha}{2}
\bigg\{ \varphi \bigg( \hspace{-0.1em} x_{\epsilon}\hspace{-0.1em}+\hspace{-0.1em} \nu_{\min} , t_{\epsilon}-\frac{\epsilon^{2}}{2}  \bigg) \hspace{-0.3em}+ \hspace{-0.3em} \varphi \bigg( \hspace{-0.1em} x_{\epsilon}\hspace{-0.1em}-\hspace{-0.1em} \nu_{\min} , t_{\epsilon}-\frac{\epsilon^{2}}{2}  \bigg) \hspace{-0.2em} \bigg\}
\\ & \qquad +\beta \kint_{B_{\epsilon}^{\nu_{\min}} } \varphi \bigg(x_{\epsilon}+ h,t_{\epsilon}-\frac{\epsilon^{2}}{2} \bigg) d \mathcal{L}^{n-1}(h) 
\\ & \ge  \varphi(x_{\epsilon}, t_{\epsilon}) - \frac{\epsilon^{2}}{2}  \varphi_{t}(x_{\epsilon}, t_{\epsilon}) 
\\ & \qquad + \frac{\beta}{2(n+1)} \epsilon^{2} \big\{  \Delta_{\nu_{\min}^{\perp}} \varphi (x_{\epsilon}, t_{\epsilon}) + (p-1) \langle D^{2} \varphi (x_{\epsilon}, t_{\epsilon}) \nu_{\min}, \nu_{\min} \rangle \big\}
.
\end{align*}
Then by \eqref{2est1}, we observe that
\begin{align} \begin{split} \label{2est2}
\frac{\epsilon^{2}}{2} \varphi_{t}(x_{\epsilon}, t_{\epsilon}) & \ge   
\frac{\beta \epsilon^{2} }{2(n+1)} \big\{  \Delta_{\nu_{\min}^{\perp}} \varphi (x_{\epsilon}, t_{\epsilon}) + (p-1) \langle D^{2} \varphi (x_{\epsilon}, t_{\epsilon}) \nu_{\min}, \nu_{\min} \rangle \big\}  .
\end{split}
\end{align}

Suppose that $D\varphi (x,t) \neq 0 $. Since $(x_{\epsilon}, t_{\epsilon}) \to (x,t)$ as $\epsilon \to 0$, it can be seen that
$$\nu_{\min} \to -\frac{D\varphi (x,t)}{|D\varphi (x,t)|} =: -\mu $$
as $\epsilon \to 0$.
We also check that 
$$\Delta_{(-\mu)^{\perp}} \varphi (x,t) + (p-1) \langle D^{2} \varphi (x_{\epsilon}, t_{\epsilon}) (-\mu), (-\mu) \rangle =\Delta_{p}^{N} \varphi (x,t)   .$$
Now we divide both side in \eqref{2est2} by $\epsilon^{2}$ and let $\epsilon \to 0$.
Since $Q_{R} \subset \Omega_{T}$, it can be seen that $ \delta(x_{\epsilon}, t_{\epsilon}) \epsilon^{-2} \to 0$ as $ \epsilon \to 0$. 
Hence, we deduce
$$ \varphi_{t}(x,t) \ge \frac{1}{n+p}  \Delta_{p}^{N} \varphi (x,t)   .$$

Next consider the case $D\varphi (x,t) = 0  $.
Observe that 
\begin{align*}
& \Delta_{\nu_{\min}^{\perp}} \varphi (x_{\epsilon}, t_{\epsilon}) + (p-1) \langle D^{2} \varphi (x_{\epsilon}, t_{\epsilon}) \nu_{\min}, \nu_{\min} \rangle \\ &
=  \Delta  \varphi (x_{\epsilon}, t_{\epsilon}) + (p-2) \langle D^{2} \varphi (x_{\epsilon}, t_{\epsilon}) \nu_{\min}, \nu_{\min} \rangle .
\end{align*}
For $p \ge 2$,  we see
\begin{align*} (p-2) \langle D^{2} \varphi (x_{\epsilon}, t_{\epsilon}) \nu_{\min}, \nu_{\min} \rangle \ge (p-2) \lambda_{\min}(D^{2} \varphi (x_{\epsilon}, t_{\epsilon})).
\end{align*}
We already know that $(x_{\epsilon}, t_{\epsilon}) \to (x,t)$ as $\epsilon \to 0$ and the map $z  \mapsto \lambda_{\min}(D^{2}\varphi(z)) $ is continuous.
Therefore, it turns out
\begin{align} \label{2est3} \varphi_{t}(x,t) \ge \frac{1}{n+p} \big\{ \Delta\varphi(x,t) + (p-2)\lambda_{\min}(D^{2}\varphi(x,t) ) \big\}
\end{align}
by similar calculation in the previous case.

For $1<p<2$, by using similar argument in the previous case and
\begin{align*} (p-2) \langle D^{2} \varphi (x_{\epsilon}, t_{\epsilon}) \nu_{\min}, \nu_{\min} \rangle & \ge (p-2) \lambda_{\max}(D^{2} \varphi (x_{\epsilon}, t_{\epsilon})) \\
& = \lambda_{\min}( (p-2)D^{2} \varphi (x_{\epsilon}, t_{\epsilon})), 
\end{align*}
we also obtain the inequality \eqref{2est3}.

We can also prove the reverse inequality to consider a function $\varphi $ touching $u$ from above and a vector $\nu_{\max} $ maximizing $  \mathscr{A}_{\epsilon} \varphi (x_{\epsilon}, t_{\epsilon}; \cdot) $
and to do similar calculation again as above.
\end{proof}

\bibliographystyle{alpha}

\end{document}